\documentclass[10pt,a4paper,reqno,oneside]{amsart}
\usepackage{tabls}
\usepackage{url}
\usepackage[linktocpage = true]{hyperref}
\hypersetup{
 colorlinks = true,
 citecolor = blue,
}
\usepackage{color}
\definecolor{red}{rgb}{1,0,0}
\definecolor{green}{rgb}{0,1,0}
\definecolor{blue}{rgb}{0,0,1}
\definecolor{refkey}{gray}{.625}
\definecolor{labelkey}{gray}{.625}
\usepackage[lite]{amsrefs}
\usepackage[a4paper]{geometry}
\usepackage{amsfonts}
\usepackage{amssymb}
\usepackage{mathrsfs}
\usepackage{amsmath}
\usepackage{amsthm}
\usepackage{amscd}
\usepackage{enumerate}
\usepackage{paralist}
\usepackage{xypic}
\usepackage{graphicx}
\usepackage{mathtools}
\usepackage[all,cmtip]{xy}
\usepackage{kantlipsum}
\usepackage{stmaryrd}
\usepackage{extarrows}
\usepackage{verbatim}
\allowdisplaybreaks

\makeatletter
 \def\title@font{\normalsize\bfseries}
 \let\ltx@maketitle\@maketitle
 \def\@maketitle{\bgroup%
 \let\ltx@title\@title%
 \def\@\title{\resizebox{\textwidth}{!}{%
  \mbox{\title@font\ltx@title}%
 }}%
 \ltx@maketitle%
 \egroup}
\makeatother

\theoremstyle{plain}
\newtheorem*{zorn*}{Zorn's lemma}
\newtheorem*{tychonoff*}{Tychonoff's theorem}
\newtheorem{definition}{Definition}[section]

\newtheorem{Cor}[definition]{Corollary}
\newtheorem{Thm}[definition]{Theorem}
\newtheorem*{theorem*}{Theorem}
\newtheorem{Def}[definition]{Definition}
\newtheorem{def-prop}[definition]{Definition-Proposition}
\newtheorem{prop}[definition]{Proposition}
\newtheorem{prop-def}[definition]{Proposition-Definition}
\newtheorem{Ex}[definition]{Example}
\newtheorem{Rem}[definition]{Remark}

\DeclareMathOperator{\id}{id}
\DeclareMathOperator{\Span}{span}

\DeclareMathOperator{\CDO}{CDO}

\DeclareMathOperator{\D}{\mathcal{D}}

\newcommand{\XX}{\mathfrak{X}}

\newcommand{\R}{\mathbb{R}}

\newcommand{\circnabla}{\overset{\nabla}{\circ}}

\newcommand{\rnabla}{R^\nabla  }
\newcommand{\cnabla}{C^\nabla }
\newcommand{\circn}{\circ_n}
\newcommand{\AAstar}{A\oplus A^*}
\newcommand{\MMstar}{TM\oplus T^*M}
\newcommand{\MM}{TM\oplus TM}

\newcommand{\DeltaTM}{\Delta^{TM}}

\newcommand{\rhoA}{a_A}

\newcommand{\anchorn}{a_A}

\newcommand{\dA }{\mathrm{d}_A}

\newcommand{\dAstar}{\mathrm{d}_*}
\newcommand{\dd}{\mathrm{d}}

\newcommand{\derX}{\widehat{X}}
\newcommand{\derx}{\widehat{x}}
\newcommand{\derad}{\widehat{\xi}}
\newcommand{\derY}{\widehat{Y}}
\newcommand{\dery}{\widehat{y}}
\newcommand{\dera}{\widehat{v}}
\newcommand{\derZ}{\widehat{Z}}
\newcommand{\derz}{\widehat{z}}
\newcommand{\dere}{\widehat{u}}

\newcommand{\DA}{\CDO(A)}

\newcommand{\DE}{\CDO(E)}

\newcommand{\langg}[2]{\langle #1 \mid #2 \rangle}

\newcommand{\langr}[2]{g( #1 , #2 )}
\newcommand{\Langr}[2]{g\big( #1 , #2 \big)}

\newcommand{\lang}[2]{\langle #1 | #2 \rangle}
\newcommand{\Lang}[2]{\big\langle #1 \big| #2 \big\rangle}

\newcommand{\Br}[2]{[ #1 , #2 ]}
\newcommand{\BrM}[2]{[ #1 , #2 ]}
\newcommand{\BrL}[2]{[ #1 , #2 ]_{L}}
\newcommand{\BrCDO}[2]{[ #1 , #2 ]}

\newcommand{\Brn}[2]{[ #1 , #2 ]^{\nabla}}
\newcommand{\BrA}[2]{[ #1 , #2 ]_A}

\newcommand{\BrAstar}[2]{[ #1 , #2 ]_*}
\newcommand{\BrAn}[2]{[#1,#2]_A}

\newcommand{\nM}{\nabla^{TM}}
\newcommand{\nMstar}{\nabla^{T^*M}}

\newcommand{\nA}{\nabla^{A}}

\newcommand{\nnAA}{\nabla^{AA}}
\newcommand{\nnAAstar}{\nabla^{AA^*}}
\newcommand{\nnAstarA}{\nabla^{A^*A}}
\newcommand{\nnAstarAstar}{\nabla^{A^*A^*}}

\newcommand {\emptycomment}[1]{}
\newcommand{\tobefilledin}{\,\cdot\,}

\newcommand{\equalbyreason}[1]{\xlongequal[]{\mbox{#1}}}

\newcommand{\CinfM}{C^\infty(M)}
\newcommand{\secM}{\Gamma(TM)}

\newcommand{\secMMstar}{\Gamma(TM\oplus T^*M)}
\newcommand{\secMM}{\Gamma(TM\oplus TM)}
\newcommand{\secA}{\Gamma(A)}
\newcommand{\secAstar}{\Gamma(A^*)}
\newcommand{\secAAstar}{\Gamma(A\oplus A^*)}
\newcommand{\secE}{\Gamma(E)}

\begin{document}
\title{From   $n$-systems to Lie   and Courant algebroids}

\author{Liqiang Cai}
\address{School of Mathematics and Statistics, Henan University} 
\email{\href{mailto:cailiqiang@vip.henu.edu.cn}{cailiqiang@vip.henu.edu.cn}}

\author{Zhixiong Chen}
\address{College of Science, China Agricultural University}
\email{\href{mailto: zxchen@cau.edu.cn}{zxchen@cau.edu.cn}}


\author{Zhuo Chen}
\address{Department of Mathematical Sciences, Tsinghua University} 
\email{\href{mailto:chenzhuo@tsinghua.edu.cn}{chenzhuo@tsinghua.edu.cn}}

\author{Yanhui Bi$^*$}
\address{Center for Mathematical Sciences, College of Mathematics and Information Science,
Nanchang Hangkong University}
\email{\href{mailto: biyanhui0523@163.com}{\textrm{(corresponding author)} biyanhui0523@163.com}}

\begin{abstract}
This paper introduces a method for constructing pure algebroids, dull algebroids, and Lie algebroids. The construction relies on what we call $n$-systems on vector bundles, and we provide explicit formulas for all resulting structure maps. Analogously, metric $n$-systems defined on metric vector bundles allow us to construct metric algebroids, pre-Courant algebroids, and Courant algebroids.\end{abstract}

\maketitle

\tableofcontents
\section{Introduction}
\subsection{The motivation}
Lie algebroids, initially introduced by Pradines \cite{PR}, furnish a rigorous framework for describing the infinitesimal structure of Lie groupoids. The reader who wishes to pursue the topic of Lie algebroids and groupoids is referred to Mackenzie’s  book \cite{Ma2} (see also \cite{CW,Ma1}) for foundational materials. In short, a Lie algebroid $(A, \BrA{\tobefilledin}{\tobefilledin}, a_A)$ consists of a vector bundle $A$ (over a base manifold $M$), a Lie bracket $\BrA{\tobefilledin}{\tobefilledin}$ on the section space $\Gamma(A)$, and an anchor map $a_A: A\to TM$, where the bracket and the anchor satisfy the Leibniz rule. 
Furthermore, any Lie algebroid $(A, \BrA{\tobefilledin}{\tobefilledin},a_A)$ admits an equivalent characterization as a homological vector field of degree (+1) on the graded manifold $(M, \wedge^\bullet A^*)$ (see \cite{Vaintrob}). Thus, Lie algebroids play an important role in  algebraic geometry \cite{Br,CRB,CB} and differential geometry \cite{PM}, while also having some applications in analytical mechanics \cite{GG,GGU}. 

Courant algebroids, introduced in \cite{LWX}, unify  Lie bialgebras and the bracket on $TM\oplus T^*M$ first discovered by Courant \cite{Courant}, where $M$ is an arbitrary smooth manifold. Around the same time and developed independently, the Dorfman bracket \cite{Dor}, introduced in the context of Hamiltonian structures, is an antisymmetrization of  the Courant bracket. Roytenberg provided an equivalent definition using the Dorfman bracket \cite{R}, which gives the relation of Courant algebroids to $L_\infty$-algebras \cite{R1}. In this paper, we use the notation  $(E, \lang{\cdotp}{\cdotp}, \circ, a_E)$ to denote a Courant algebroid. Specifically, we have a pseudo-Euclidean vector bundle $(E, \lang{\cdotp}{\cdotp})$ (over a base manifold $M$), equipped with a vector bundle morphism $a_E: E \to TM$, also called the anchor, and a Leibniz algebra structure $\circ$, called the Dorfman bracket, defined on the section space $\Gamma(E)$, subject to certain compatibility conditions. The subject of Courant algebroids has inspired significant research in higher structure \cite{BR}, Poisson geometry \cite{BCSX,BP,CS,GMX,MX} and generalized geometry \cite{VCLD,MGF}, while also having some important applications in 3-dimensional topological field theory \cite{HP,I1,I2}, string theory \cite{De} and double field theory \cite{Vaisman}. 

In a prior work \cite{CCLX}, we presented a  construction of split Courant algebroids utilizing Dirac generating operators. Although the general theory is well developed, explicit structure maps for the bundles one encounters in practice are often not easy to write down. Therefore, it is pertinent to seek concrete examples of these algebraic objects where their structures can be explicitly formulated and analyzed. This paper addresses this need by introducing the concept of $n$-systems and demonstrating how these $n$-systems can be used to generate concrete examples of both Lie algebroids and Courant algebroids.
\subsection{Main results and organization}
An anchored bundle, denoted $(A, a_A)$, is the fundamental object of study in Section \ref{sec 2}. It consists of a vector bundle $A$ over a base manifold $M$, equipped with a bundle map $a_A: A \to TM$, termed the anchor. At this stage, $(A, a_A)$ lacks a Lie structure. We aim to investigate the conditions under which a Lie bracket, denoted by $[\cdot, \cdot] $, can be defined on $\Gamma(A)$, the space of sections of $A$, such that $(A, [\cdot, \cdot] , a_A)$ constitutes a Lie algebroid. Our approach involves employing connections on the anchored bundle, drawing inspiration from the role of connections in differential geometry on smooth manifolds (see Definition \ref{def-connection}). We denote such a connection on $(A, a_A)$ by $\nabla$, and we seek a bracket on $\Gamma(A)$ of the torsion-free form:
\begin{equation}\label{torsionfree}
	[x,y]^\nabla=\nabla_xy-\nabla_yx,\quad \forall x,y\in \Gamma(A).
\end{equation}
Indeed, it can be demonstrated that any Lie algebroid structure compatible with $(A, a_A)$ must adopt this form (see  Proposition \ref{pure-connection-correspondence}). In this case, the connection $\nabla$ is constrained by a condition referred to as the Lie-Bianchi identity (see Proposition \ref{Lie-Bianchi identity}).

The primary contribution of this paper is the introduction of a method for constructing  connections and their associated Lie algebroid structures. This construction relies on a specific set of data, termed an ``$n$-system" on $A$, denoted by $(\derX_1,\ldots,\derX_n;\derad^1,\ldots,\derad^n)$. Here, $\derX_1$, $\ldots$, $\derX_n$ are covariant differential operators acting on $A$, while $\derad^1$, $\ldots$, $\derad^n$ are elements of $\secAstar$, linearly independent over $\CinfM$. From a given $n$-system, we define an anchor map $a_A: A\to TM$ by
$$\anchorn(y)  =  \Lang{y}{\derad^i} \cdot \derx_i, \forall y\in \Gamma(A),  $$
where $\derx_i\in \Gamma(TM)$ represents the symbol of $\derX_i$. Simultaneously, a connection $\nA$ is defined as:
$$
\nA_yz := \Lang{y}{\derad^i} \cdot \derX_i(z),\hspace*{2em} \forall y, z \in \secA.
$$
Consequently, the bracket operation on $\Gamma(A)$, defining the Lie algebroid structure on $A$, is given explicitly by:
\begin{eqnarray*}  \BrAn{y}{z} &=& \nA_yz - \nA_zy = \Lang{y}{\derad^i} \cdot \derX_i(z) - \Lang{z}{\derad^i} \cdot \derX_i(y),\hspace*{2em}\forall y, z \in \secA. 
\end{eqnarray*}

We note that the Lie algebroid structure of $T^*M$ constructed in \cite{DobJak} is a special case derived from a particular $1$-system. 

The satisfaction of  Lie algebroid axioms by the aforementioned structures on $A$ requires additional conditions for the said $n$-system. Proposition \ref{Lie algebroid} provides a sufficient condition, expressed by Equations \eqref{Eqn:compatible condition 1 of Lie algebroid in n system} and \eqref{Eqn:compatible condition 2 of Lie algebroid in n system}. When these two equations hold, the $n$-system $(\derX_1,\ldots,\derX_n;\derad^1,\ldots,\derad^n)$ is termed a ``compatible $n$-system".

In conclusion, any compatible $n$-system gives rise to an associated Lie algebroid with explicitly defined structure maps. Moreover, the language of compatible $n$-systems allows us to characterize certain DG manifolds and Lie pairs, as demonstrated in Sections \ref{DG} and \ref{Lie pair}. In Section \ref{special type example}, we further investigate a special type of compatible $n$-system, providing a characterization of its associated Lie algebroid in terms of explicit bases and functions, alongside concrete examples.

Analogously to Lie algebroids, Courant algebroids can  be constructed using a metric connection. In the second part of this paper, we will elaborate on this construction, beginning with an arbitrary anchored metric bundle $(E, \lang{\cdotp}{\cdotp}, a_E)$ over a base manifold $M$. Here, $(E, a_E)$ denotes an anchored bundle equipped with a metric $\lang{\cdotp}{\cdotp}$ on $E$. We aim to utilize connections on this anchored metric bundle, specifically metric connections, denoted by $\nabla$ (see Section \ref{sec 3.1}). These connections are required to be compatible with the metric $\lang{\cdotp}{\cdotp}$. 

Once the induced bracket is Courant, such a metric connection is a generalized connection in the sense of \cites{CPR}. The authors also study flat generalized connections on transitive Courant algebroids using the mixed components of the curvature determined by a generalized metric \cite{CPRFlat}. On a Courant algebroid, our operator $C^\nabla$ is their na\"ive curvature $\mathcal R_0$. Here we start from an anchored metric bundle and use compatible metric $n$-systems to construct the Dorfman bracket and a connection for which the Courant curvature $C^\nabla$ vanishes.

A key distinction between the Lie bracket and the Dorfman bracket lies in the skew-symmetry property, which is satisfied by the former but not by the latter. Consequently, a Dorfman bracket on $\Gamma(E)$ cannot be expressed in the form of Equation \eqref{torsionfree}. Instead, we will demonstrate that the Dorfman bracket necessarily takes the form:
\[
e_1\circnabla e_2=\nabla_{e_1}e_2-\nabla_{e_2}e_1+\Delta_{e_2}e_1,
\]
where the operator $\Delta:~\Gamma(E)\times\Gamma(E)\rightarrow\Gamma(E)$ is determined by $\nabla$ according to the relation:
\[\langle\Delta_{e_2}e_1, e_3\rangle:=\langle\nabla_{e_3}e_1, e_2\rangle,\quad\forall e_1,e_2,e_3\in\Gamma(E).\]
Proposition \ref{prop:surj-Metric algebroid} provides a proof of this result. Following this, Section \ref{three typical examples} examines three specific examples of anchored metric bundles, analyzing their metric, pre-Courant, and Courant algebroid structures.

To construct metric connections and their associated Courant algebroids, we introduce a method analogous to our previous construction of Lie algebroids from $n$-systems for vector bundles. Specifically, given a metric vector bundle $(E, \lang{\cdotp}{\cdotp})$, we define a  \emph{metric $n$-system}  as a tuple $(\derZ_1,\ldots,\derZ_n;\dere^1,\ldots,\dere^n)$, where $\derZ_1, \ldots, \derZ_n$ are metric-compatible covariant differential operators on $E$, and $\dere^1, \ldots, \dere^n$ are elements of $\Gamma(E)$ that are linearly independent over $\CinfM$. From such a metric $n$-system, we construct a metric connection $\nabla$ and an operator $\Delta$ on $\Gamma(E)$ as follows:
\[ \nabla_{e_1}e_2
:=\Lang{e_1}{\dere^i}\cdot\derZ_i(e_2), \quad \Delta_{e_2}e_1:=\Lang{\derZ_i(e_1)}{e_2}\cdot\dere^i,  \]
for all $e_1, e_2 \in \Gamma(E)$. Consequently, we formulate the associated Courant algebroid structure on $E$ as follows:
\begin{itemize}
	\item The anchor map $a_E: E \to TM$ is defined by $a_E(e) = \Lang{e}{\dere^i} \cdot \derz_i$, and the operator $\D: C^\infty(M) \to \Gamma(E)$ is given by $\D(f) := \derz_i(f) \cdot \dere^i$, where $\derz_i\in \Gamma(TM)$ denotes the symbol of $\derZ_i$;
	\item The Dorfman bracket $\circnabla: \Gamma(E) \times \Gamma(E) \to \Gamma(E)$ is defined by
	\[e_1\circnabla e_2
	=\Lang{e_1}{\dere^i}\cdot\derZ_i(e_2)
	-\Lang{e_2}{\dere^i}\cdot\derZ_i(e_1)
	+\Lang{\derZ_i(e_1)}{e_2}\cdot \dere^i,\]
\end{itemize}
for all $f\in C^\infty(M)$ and $e,e_1,e_2\in \Gamma(E)$.

To ensure that the aforementioned structures satisfy the axioms of a Courant algebroid, we invoke Proposition \ref{Courant algebroid}, which provides a sufficient condition. This condition is expressed by the three equations: \eqref{Eqn: compatible condition 1 of Courant algebroid in n system}, \eqref{Eqn: compatible condition 3 of Courant algebroid in n system}, and \eqref{Eqn: compatible condition 2 of Courant algebroid in n system}. When these three equations hold, the system $(\derZ_1,\ldots,\derZ_n;\dere^1,\ldots,\dere^n)$ is defined as a compatible metric $n$-system. Consequently, given any such compatible metric $n$-system, we can conclude that there exists an associated Courant algebroid, the structure maps of which can be explicitly formulated.

In Sections \ref{Lie bialgebroids from metric} and \ref{Manin pairs and dorfman connections}, we present applications of our results. Specifically, Section \ref{Lie bialgebroids from metric} focuses on a particular class of compatible metric $n$-systems, characterizing the Courant algebroid structures they induce and providing illustrative examples. Furthermore, Section \ref{Manin pairs and dorfman connections} demonstrates the utility of compatible metric $n$-systems in characterizing certain Manin pairs and Dorfman connections.

\textbf{Acknowledgements.} 
We would like to thank Honglei Lang, Zhangju Liu, Ping Xu, and Maosong Xiang    for fruitful discussions and useful comments.
This research is partially supported by NSFC grants 12661050, 11961049, 10601219, 12071241, 11701146.
\subsection{List of terminologies  and notations} \label{notation}
\begin{enumerate}
	\item   
	Throughout this paper, we adopt the Einstein convention, i.e., $a^ib_i$ implies summation over the index $i$.
	
	\item The letter $A$ (or $E$) denotes a vector bundle over a smooth manifold $M$. 
	
	\item 
	We define $\Gamma(\DA)$ as the space of covariant differential operators on $A$. Specifically, an element $\derX \in \Gamma(\DA)$ is an $\R$-linear operator $\secA \rightarrow \secA$, together with vector filed $\derx \in \secM$, called its symbol, such that for all $y \in \secA$ and $f \in \CinfM$, the condition $$\derX(fy) = f\derX(y) + \derx(f)\cdot y$$ holds.
	
	\item \textit{The standard pairing} on $\AAstar$ is defined for all $x+\xi, y+\eta \in \secAAstar$ as \begin{eqnarray}
		\label{Eqn:standard pairing}
		\lang{x+\xi}{y+\eta}:=\xi(y)+\eta(x).
	\end{eqnarray}
	
	\item \textit{A pure algebroid} $(A, \BrA{\tobefilledin}{\tobefilledin}, \rhoA)$ is composed by a vector bundle $A \rightarrow M$, a bundle map $\rhoA: A \rightarrow TM$ known as the anchor map, and a skew-symmetric bracket $\BrA{\tobefilledin}{\tobefilledin}$ on $\secA$ satisfying the Leibniz identity \begin{eqnarray*}
		\BrA{x}{fy}
		=f\BrA{x}{y}
		+\rhoA(x)(f)\cdot y,
		\qquad\forall x,y\in\secA, f\in\CinfM.
	\end{eqnarray*}
	
	\item \textit{A dull algebroid}, which we define slightly differently from \cite{MJ2018} (where the skew-symmetric property of $\BrA{\tobefilledin}{\tobefilledin}$ is not required), is a pure algebroid $(A, \BrA{\tobefilledin}{\tobefilledin}, \rhoA)$ that satisfies an additional condition: \begin{eqnarray}
		\label{Eqn:dull algebroid anchor condition}
		\rhoA(\BrA{x}{y})=\Br{\rhoA(x)}{\rhoA(y)},
		\qquad\forall x,y\in\secA.
	\end{eqnarray} 
	\item \textit{A Lie algebroid} is a pure algebroid $(A, \BrA{\tobefilledin}{\tobefilledin}, \rhoA)$ that adheres to the Jacobi identity \begin{eqnarray*}
		\BrA{\BrA{x}{y}}{z}
		+\BrA{\BrA{y}{z}}{x}
		+\BrA{\BrA{z}{x}}{y}
		=0,
		\qquad\forall x,y,z\in\secA.
	\end{eqnarray*} 
	Notably, a Lie algebroid must be a dull algebroid, as the Jacobi identity implies Condition \eqref{Eqn:dull algebroid anchor condition}.
	
	\item \textit{A metric algebroid} $(E, \lang{\tobefilledin}{\tobefilledin}, \circ, a_E)$, as discussed in \cite{Vaisman}, is a vector bundle $E \rightarrow M$ equipped with three structures: a pseudo-metric $\lang{\tobefilledin}{\tobefilledin}$ on $\secE$, which is a non-degenerate symmetric bilinear form on $\secE$; a bilinear operation $\circ$ on $\secE$ known as the Dorfman bracket; and a bundle map $a_E: E \rightarrow TM$ called the anchor. These structures must satisfy the following axioms: for all $e, h_1, h_2 \in \secE$,
	\begin{itemize}
		\item[$(i)$]
		$a_E(e)  \lang{h_1}{h_2}
		=\lang{ e \circ h_1}{h_2}
		+\lang{ h_1}{e \circ h_2}$,
		\item[$(ii)$]
		$e \circ e=\frac{1}{2}\D\lang{ e}{e}$,
	\end{itemize} where $\D: \CinfM \rightarrow \secE$ is determined by $\lang{ \D(f)}{e} = a_E(e)(f)$ for all $f \in \CinfM$. An important equality for metric algebroids, used in \cite{R}, is \begin{eqnarray}\label{Eqn:Metric algebroid D and anchor condition}
		\Lang{\D(f)\circ  e}{h}
		=\big(a_E\big(e\circ  h\big)-\BrM{a_E(e)}{a_E(h)}\big)(f),
	\end{eqnarray}
	for all $e,h\in\secE, f\in\CinfM$.
	
	\item \textit{A pre-Courant algebroid}, as defined in \cite{Vaisman-2}, is a metric algebroid $(E, \lang{\tobefilledin}{\tobefilledin}, \circ, a_E)$ that satisfies an additional condition: \begin{eqnarray}\label{Compat-cond-pre-Courant}
		a_E(e_1\circ e_2)=\Br{a_E(e_1)}{a_E(e_2)},
		\qquad\forall e_1,e_2\in\secE.
	\end{eqnarray}
	According to Equation \eqref{Eqn:Metric algebroid D and anchor condition}, this condition also implies $\D(f)\circ e = 0$ for all $e \in \secE$ and $f \in \CinfM.$ 
	
	\item  \textit{A Courant algebroid} is a metric algebroid  $(E, \lang{\tobefilledin}{\tobefilledin}, \circ, a_E)$ that satisfies the Leibniz identity \begin{eqnarray*}\label{metric algebroid to Courant}
		e_1\circ(e_2\circ e_3)
		=(e_1\circ e_2)\circ e_3
		+e_2\circ(e_1\circ e_3),
		\qquad\forall e_1,e_2,e_3\in\secE.
	\end{eqnarray*}
\end{enumerate}
\section{From connections to Lie algebroids}\label{sec 2}
This section introduces compatible $n$-systems on vector bundles. We present our main findings, which characterize pure algebroids, dull algebroids, and Lie algebroids through connections and $n$-systems.
\subsection{Connections and pure algebroids}\label{Sec:2-1}
We begin by outlining some key concepts and facts that are well established in the literature \cite{PM}. 
Recall that an \textbf{anchored vector bundle} $(A,a_A)$ is a vector bundle $A\to M$
with an \textbf{anchor $a_A\colon A\to TM$}.

\begin{Def}\label{def-connection}
	A \textbf{connection} on $(A,a_A)$ is an $\mathbb{R}$-bilinear operator
	$\nabla\colon\Gamma(A)\times\Gamma(A)\to\Gamma(A)$ satisfying
	\[
	\nabla_{fx}y=f\nabla_x y,\qquad
	\nabla_x(fy)=a_A(x)(f)\cdot y+f\nabla_x y
	\]
	for all $x,y\in\Gamma(A)$ and $f\in C^\infty(M)$.
\end{Def}
Associated to any connection $\nabla$ on $(A,a_A)$, there exists a skew-symmetry bracket:
 \begin{eqnarray}\label{pure bracket}
\Brn{\tobefilledin}{\tobefilledin}:\secA\times\secA\rightarrow\secA,\qquad\Brn{x}{y}:=\nabla_x y-\nabla_y x
 	, \qquad\forall x,y\in\secA.
 \end{eqnarray}
\begin{prop}\label{pure-connection-correspondence}
	Every connection on $(A,a_A)$ defines a pure algebroid
	$(A,[\,\cdot,\cdot\,]^\nabla,a_A)$ by Equation \eqref{pure bracket}. Conversely, every pure algebroid structure on $(A,\rhoA)$ arises from an associated connection in this way.
\end{prop}
\begin{proof} 
	We only prove the second statement.
Beginning with an arbitrary  connection $\nM:\secM\times\secA\rightarrow\secA$, we can define a corresponding operator $\nA:\secA\times\secA\rightarrow\secA$ through the following relation:
\begin{eqnarray*}
	\nA_{x}y
	:=\nM_{\rhoA(x)}y,
	\qquad\forall x,y\in\secA.
\end{eqnarray*}

To construct the desired connection, we introduce a skew-symmetric and $\CinfM$-bilinear operator $K:\secA\otimes\secA\rightarrow\secA$, defined as
\begin{eqnarray*}
	K(x,y):=\BrA{x}{y}-\nA_xy+\nA_yx,
	\qquad\forall x,y\in\secA.
\end{eqnarray*}

Subsequently, we construct an operator $\nabla:\secA\times\secA\rightarrow\secA$ by
\begin{eqnarray*}
	&&\nabla_xy
	:=\nA_xy
	+\frac{1}{2} K(x,y),
	\qquad\forall x,y\in\secA.
\end{eqnarray*}

One can verify that $\nabla$ determines a connection on $(A,\rhoA)$ and satisfies the relation:
$$\BrA{x}{y}=\nabla_xy-\nabla_yx=:[x,y]^\nabla.$$
\end{proof}
\begin{definition}\label{pure-connection equivalent}
	Two connections $\nabla$ and $\nabla'$ on $\big(A, \rhoA\big)$ are called equivalent  if there exists a symmetric bundle map $S:A \otimes A \rightarrow A$ satisfying
	$$\nabla'_x y=\nabla_x y + S(x,y),\quad \forall x,y\in \secA.$$
\end{definition}

As a direct consequence of Proposition \ref{pure-connection-correspondence} and Definition \ref{pure-connection equivalent}, we can state the following one-to-one correspondence:

\begin{prop}\label{equivalence classes of connections on anchored bundle}
	There exists a canonical one-to-one correspondence between   equivalence classes of connections on the anchored bundle $\big(A, \rhoA\big)$ and   pure algebroid structures underlying   $\big(A ,\rhoA\big)$. 
\end{prop}

Let us introduce the curvature  of $\nabla$, which is an operator
$$\rnabla :\secA\times\secA\rightarrow\Gamma(\DA)$$
 defined as
\begin{eqnarray}\label{Def: curvature}
	\rnabla (x, y)(z)
	:=\nabla_x \nabla_yz
	-\nabla_y \nabla_xz
	-\nabla_{(\Brn{x}{y})}z,
\end{eqnarray}
 for all $x,y,z\in\secA$.

According to Equation \eqref{Def: curvature}, the operator $\rnabla$ is skew-symmetric and $C^\infty(M)$-linear in the arguments $x$ and $y$. This property results in the following proposition:

\begin{prop}\label{proposition 2.5}
	For a pure algebroid $\big(A,\Brn{\tobefilledin}{\tobefilledin}, \rhoA\big)$ arising from the connection $\nabla$, the following statements are equivalent:
	\begin{itemize}
		\item The curvature $\rnabla$ is $C^\infty(M)$-linear with respect to the argument $z$,
		\item $\big(A,\Brn{\tobefilledin}{\tobefilledin}, \rhoA\big)$ is a dull algebroid.
	\end{itemize}
\end{prop}
The proof is easy and is therefore omitted. Through direct computation, we obtain another significant property:

\begin{prop}\label{pure-Bianchi identity}
	With the same condition as in Proposition \ref{proposition 2.5},
	for a pure algebroid $\big(A,\Brn{\tobefilledin}{\tobefilledin}, \rhoA\big)$, the following identity holds:
	\begin{eqnarray}\label{Eqn:pBI of pure algebroid}
		\rnabla (x,y)z
		+\rnabla (y,z)x
		+\rnabla (z,x)y
		=\Brn{x}{\Brn{y}{z}}
		+\Brn{y}{\Brn{z}{x}}
		+\Brn{z}{\Brn{x}{y}},
	\end{eqnarray}
	where $x,y,z\in\secA$.
\end{prop}

Equation \eqref{Eqn:pBI of pure algebroid} is called the \textbf{pure-Bianchi identity} (of a pure algebroid). When the bracket is a Lie bracket it reduce to the classical first Bianchi identity.

\begin{prop}\label{Lie-Bianchi identity}
	The following statements are equivalent:
	\begin{itemize}
		\item[$(1)$]
		$\big(A,\Brn{\tobefilledin}{\tobefilledin}, \rhoA\big)$ is a Lie algebroid;
		\item[$(2)$]
		For all $x,y,z\in\secA$, the following identity holds:
		\begin{eqnarray}\label{Eqn:LBI of Lie algebroid}
			\rnabla (x,y)z
			+\rnabla (y,z)x
			+\rnabla (z,x)y
			=0.
		\end{eqnarray}
	\end{itemize}
\end{prop}

We call Equation \eqref{Eqn:LBI of Lie algebroid} the \textbf{Lie-Bianchi identity} (of a Lie algebroid).
\begin{Rem}
	There is another type of curvature in \cite{ABD}. Let $M$ be a smooth manifold and $E\rightarrow M$ be a vector bundle; a connection $\nabla$ is a degree $(+1)$ derivation of the differential graded module $\Omega(M;E)=\Gamma(\wedge T^*M \otimes E)$ over the differential graded algebra $\big(\Omega(M),\wedge, \dd\big)$, that is, an operator $
		\nabla:\Gamma(\wedge^\bullet T^*M\otimes E)\rightarrow\Gamma(\wedge^{\bullet+1}T^*M\otimes E)$ satisfying
		\begin{eqnarray*}
			\nabla(\omega \sigma)
			=(\dd\omega)\cdot\sigma
			+(-1)^{|\omega|}\omega\cdot\nabla(\sigma),
			\qquad\forall \omega\in\Omega(M), \sigma\in\Omega(M;E).
		\end{eqnarray*}
		The curvature of $\nabla$ is $R_\nabla=\nabla^2$, which is defined to be
		an $\Omega(M)$-module morphism of degree $(+2)$.
		It is different from the curvature that we use in this paper.
\end{Rem}
\subsection{Lie algebroids from compatible $n$-systems}\label{sec 2.4} Consider a vector bundle $A$ over a base manifold $M$, and suppose that $\derad^1$, $\derad^2$, $\ldots$, $\derad^n$ are elements of $\secAstar$ that are linearly independent over the ring of smooth functions $\CinfM$. Furthermore, let $\derX_1$, $\derX_2$, $\ldots$, $\derX_n$ be covariant differential operators on $A$, belonging to $\Gamma(\CDO(A))$, where these operators might be linearly dependent over $\CinfM$. In the sequel, the rank of the vector bundle $A$ is at least $n$.
\begin{definition}\label{n-system on A}
	 The datum $$(\derX_1,\ldots,\derX_n;\derad^1,\ldots,\derad^n)$$ is called  an
\textbf{$n$-system} on  $A$.
\end{definition}

In the subsequent discussion, we consistently represent the symbols of $\derX_1, \ldots, \derX_{n-1}$, and $\derX_n$ by the letters $\derx_1, \ldots, \derx_{n-1}$, and $\derx_n$ (where $\derx_i \in \secM$ for each $i$). Given an $n$-system $(\derX_1, \ldots, \derX_n; \derad^1, \ldots, \derad^n)$, one can  define the following structure maps on $\Gamma(A)$, which  establish a pure algebroid $(A,\BrAn{\tobefilledin}{\tobefilledin},\anchorn)$:
\begin{itemize}
	\item The anchor map $\anchorn: A \rightarrow TM$ is defined as
	\begin{eqnarray}
		\label{Def:anchor of Lie algebroid in n system}
		\anchorn(y) := \Lang{y}{\derad^i} \cdot \derx_i,\hspace*{2em} \forall y \in \secA.
	\end{eqnarray}
	\item The connection $\nA: \secA \times \secA \rightarrow \secA$ is defined by
	\begin{eqnarray}
		\label{Def:connection of Lie algebroid in n system}
		\nA_yz := \Lang{y}{\derad^i} \cdot \derX_i(z),\hspace*{2em}\forall y, z \in \secA.
	\end{eqnarray}
	\item The bracket $\BrAn{\tobefilledin}{\tobefilledin}: \secA \times \secA \rightarrow \secA$, induced by $\nA$, is formulated as
	\begin{eqnarray}
		\label{Def:bracket of Lie algebroid in n system}
		\BrAn{y}{z} := \nA_yz - \nA_zy = \Lang{y}{\derad^i} \cdot \derX_i(z) - \Lang{z}{\derad^i} \cdot \derX_i(y),\hspace*{2em}\forall y, z \in \secA. 
	\end{eqnarray}
\end{itemize}
 For specific examples of $n$-systems, see Section \ref{special type example}.
 The next step is to investigate the conditions under which this pure algebroid can be characterized as either a dull algebroid or a Lie algebroid. The first fact we need is the following proposition.
\begin{prop}\label{dull algebroid}
	Consider the $n$-system $(\derX_1,\ldots,\derX_n;\derad^1,\ldots,\derad^n)$ on $A$, and suppose that it satisfies the following conditions: 
	\begin{eqnarray*}
		\derX_i(\derad^j) =a_{ik}^j \derad^k, 
		\label{609-2} \qquad
		\BrM{\derx_i}{\derx_j} =(a_{ji}^k-a_{ij}^k)\derx_k, 
	\end{eqnarray*}
	where $a_{ik}^j\in\CinfM$. Under these conditions, the triple $(A,\BrAn{\tobefilledin}{\tobefilledin},\anchorn)$ constitutes a dull algebroid.
\end{prop}
\begin{proof}
Based on the definitions and conditions presented above, one verifies the following relation:
\begin{eqnarray*}
\anchorn(\BrAn{x}{y})-\BrM{\anchorn(x)}{\anchorn(y)}=0,
\qquad\forall x,y\in\secA.
\end{eqnarray*}
In fact, the left-hand side of the above equation can be expressed as:
{
\begin{eqnarray*}
&&
\Lang{x}{\derad^i}\cdot\Lang{\derX_i(y)}{\derad^j}\cdot\derx_j
-\Lang{y}{\derad^i}\cdot\Lang{\derX_i(x)}{\derad^j}\cdot\derx_j
\\&&\quad
-\Lang{x}{\derad^i}\cdot\Lang{y}{\derad^j}\cdot\BrM{\derx_i}{\derx_j}
-\Lang{x}{\derad^i}\cdot\derx_i\big(\Lang{y}{\derad^j}\big)\cdot\derx_j
+\Lang{y}{\derad^j}\cdot\derx_j\big(\Lang{x}{\derad^i}\big)\cdot\derx_i
\\&\hspace*{-0.5em}=\hspace*{-0.5em}&
-\Lang{x}{\derad^i}\cdot\Lang{y}{\derad^j}\cdot\BrM{\derx_i}{\derx_j}
-\Lang{x}{\derad^i}\cdot\Lang{y}{\derX_i(\derad^j)}\cdot\derx_j
+\Lang{y}{\derad^i}\cdot\Lang{x}{\derX_i(\derad^j)}\cdot\derx_j
\\&\hspace*{-0.5em}=\hspace*{-0.5em}&
-\Lang{x}{\derad^i}\cdot\Lang{y}{\derad^j}\cdot(a_{ji}^k-a_{ij}^k)\derx_k
-\Lang{x}{\derad^i}\cdot\Lang{y}{a_{ik}^j \derad^k}\cdot\derx_j
+\Lang{y}{\derad^i}\cdot\Lang{x}{a_{ik}^j \derad^k}\cdot\derx_j
\\&\hspace*{-0.5em}=\hspace*{-0.5em}&0.
\end{eqnarray*}
}
\end{proof}
\begin{prop}\label{Lie algebroid}
	Consider an $n$-system $(\derX_1,\ldots,\derX_n;\derad^1,\ldots,\derad^n)$ on $A$, and suppose that it satisfies  the following conditions: 
	\begin{eqnarray}
		\label{Eqn:compatible condition 1 of Lie algebroid in n system}
		\derX_i(\derad^j)
		&\hspace*{-0.5em}=\hspace*{-0.5em}&a_{ik}^j \derad^k,
		\\
		\label{Eqn:compatible condition 2 of Lie algebroid in n system}
		\BrCDO{\derX_i}{\derX_j}
		&\hspace*{-0.5em}=\hspace*{-0.5em}&(a_{ji}^k-a_{ij}^k)\derX_k,
	\end{eqnarray}
	
	where $a_{ik}^j\in\CinfM$. Under these conditions, the following equation holds: $$\rnabla(x,y)z=0,$$ for all $x, y, z \in \secA$. Consequently, the triple  $(A,\BrAn{\tobefilledin}{\tobefilledin},\anchorn)$ constitutes a Lie algebroid.
\end{prop}
\begin{proof}
For all $x,y,z\in\secA$, by Formulas \eqref{Def:anchor of Lie algebroid in n system}---\eqref{Def:bracket of Lie algebroid in n system} and Conditions \eqref{Eqn:compatible condition 1 of Lie algebroid in n system}---\eqref{Eqn:compatible condition 2 of Lie algebroid in n system}, we obtain
\begin{eqnarray*}
\rnabla (x,y)z
&=&\nabla_x \nabla_yz
-\nabla_y \nabla_xz
-\nabla_{(\Brn{x}{y})}z
\\&=&\lang{x}{\derad^i}\cdot\lang{y}{\derad^j}\cdot\BrCDO{\derX_i}{\derX_j}(z)
+\big(\lang{x}{\derad^i}\cdot\Lang{y}{\derX_i(\derad^k)}
-\lang{y}{\derad^j}\cdot\lang{x}{\derX_j(\derad^k)}\big)\cdot\derX_k(z)
\\&=&\lang{x}{\derad^i}\cdot\lang{y}{\derad^j}\cdot(a_{ji}^k-a_{ij}^k)\derX_k(z)
+\big(\lang{x}{\derad^i}\cdot\Lang{y}{a_{ij}^k\derad^j}
-\lang{y}{\derad^j}\cdot\lang{x}{a_{ji}^k\derad^i}\big)\cdot\derX_k(z)
\\&=&
0
.
\end{eqnarray*}
Moreover, by Equation \eqref{Eqn:pBI of pure algebroid}, we have
$
\BrA{x}{\BrA{y}{z}}
+\BrA{y}{\BrA{z}{x}}
+\BrA{z}{\BrA{x}{y}}
=
\rnabla (x,y)z
+\rnabla (y,z)x
+\rnabla (z,x)y
=0
$.
\end{proof}
\begin{Def}
	An $n$-system $(\derX_1, \ldots, \derX_n; \derad^1, \ldots, \derad^n)$ on $A$ is called \textbf{compatible} if it satisfies Conditions \eqref{Eqn:compatible condition 1 of Lie algebroid in n system} and \eqref{Eqn:compatible condition 2 of Lie algebroid in n system}.
\end{Def}
Consider the simplest case where $n=1$, which corresponds to a $1$-system on the vector bundle $A\to M$, say $\big(\derX;\derad~\big)$ where $\derX\in \Gamma(\CDO(A))$, $\derad\in \secAstar$. The compatibility becomes
\begin{eqnarray*}
	\derX(\derad)=\hat{a}\derad,\qquad\text{for some}\quad \hat{a}\in\CinfM.
\end{eqnarray*}
This condition implies that $\derad$ can be regarded as an eigenvector of $\derX$. Arising from this $1$-system, the corresponding Lie algebroid structure on $A$ can be explicitly characterized through the following relations:
\begin{eqnarray*}
	\rhoA(y)
	&\hspace*{-0.5em}=\hspace*{-0.5em}&\Lang{y}{\derad}\cdot\derx
	,\\
	\BrA{y}{z}
	&\hspace*{-0.5em}=\hspace*{-0.5em}&\Lang{y}{\derad}\cdot\derX(z)
	-\Lang{z}{\derad}\cdot\derX(y)
	,\qquad\forall y,z\in\secA.
\end{eqnarray*}
\subsection{A special type of compatible $n$-systems}\label{special type example}
Let \( I \subset \mathbb{R} \) be an open subset on which smooth functions are represented as \( f = f(t) \), with \( t \) being the coordinate of \( I \). Consider the \( m \)-dimensional vector space \( V = \mathrm{Span}_{\mathbb{R}}\{w_1, w_2, \ldots, w_m\} \) and its dual space \( V^* = \mathrm{Span}_{\mathbb{R}}\{w^1, w^2, \ldots, w^m\} \). We define two product vector bundles \( A = I \times V \) and \( A^* = I \times V^* \) over $I$.

Let \(\derX_1, \derX_2, \ldots, \derX_n \in \Gamma(\mathrm{CDO}(A))\) ($n\leqslant m$) be covariant differential operators, each associated with a symbol \(\derx_i = f_i \frac{d}{dt}\), where $f_i\in C^\infty(I)$. The action of each operator \(\derX_i\) is specifically characterized by its operation on the basis vectors \( w^1,\cdots,w^m \) of \( V^* \), which are treated as constant sections in \(\secAstar\). Suppose that this action is explicitly expressed as:
\begin{eqnarray*}
	\derX_i\begin{pmatrix}w^1\\w^2\\\vdots\\w^m\end{pmatrix}
	=A_i\begin{pmatrix}w^1\\w^2\\\vdots\\w^m\end{pmatrix}
	=\begin{pmatrix}
		a_{i1}^1 &a_{i2}^1&\cdots&a_{im}^1\\
		a_{i1}^2&a_{i2}^2&\cdots&a_{im}^2\\
		\cdots&\cdots&\cdots&\cdots\\
		a_{i1}^m&a_{i2}^m&\cdots&a_{im}^m\\
	\end{pmatrix}
	\begin{pmatrix}w^1\\w^2\\\vdots\\w^m\end{pmatrix},
\end{eqnarray*}

where each $A_i$ ($i=1,\ldots,n$) is an $m\times m$ matrix with entries functions $a_{ik}^j=a_{ik}^j(t)\in C^\infty(I)$ ($1\leqslant j,k\leqslant m$). The collection $(\derX_1$, $\ldots$, $\derX_n$;  $w^1$, $\ldots$, $w^n)$ forms an $n$-system on $A$.

In what follows we only let the indices $i,j,k$ range from $1$ to $n$. The compatibility condition $\BrM{\derx_i}{\derx_j}=(a_{ji}^k-a_{ij}^k)\derx_k$ can be expressed  as:
\begin{eqnarray}\label{Eqn:compatible condition  of dull algebroid in special n system}
	\BrM{f_i\frac{d}{dt}}{f_j\frac{d}{dt}}=(a_{ji}^k-a_{ij}^k)f_k\frac{d}{dt}
	\quad\text{or}\quad
	f_if'_j-f_jf'_i=(a_{ji}^k-a_{ij}^k)f_k.
\end{eqnarray}

Furthermore, the condition $\BrCDO{\derX_i}{\derX_j}=(a_{ji}^k-a_{ij}^k)\derX_k$ implies Equation \eqref{Eqn:compatible condition  of dull algebroid in special n system} and the following equation:
\begin{eqnarray}\label{Eqn:compatible condition 2 of Lie algebroid in special n system}
	A_jA_i-A_iA_j+f_iA'_j-f_jA'_i
	=(a_{ji}^k-a_{ij}^k)A_k.
\end{eqnarray}
When $f_i$ and $A_i$ satisfy Equations \eqref{Eqn:compatible condition  of dull algebroid in special n system} and \eqref{Eqn:compatible condition 2 of Lie algebroid in special n system}, $(\derX_1,\ldots,\derX_n; w^1,\ldots,w^n)$ constitutes a compatible $n$-system. The associated Lie algebroid structure on $A=I\times V$ is explicitly given by
$$
\anchorn(w_i)=f_i\frac{d}{dt},\quad \anchorn(w_{n+1})=\cdots=\anchorn(w_{m})=0,
$$
and
$$
\BrAn{w_i}{w_j}
=\derX_i(w_j)-\derX_j(w_i)
=(a_{ji}^k-a_{ij}^k)w_k,
$$
$$
\BrAn{w_i}{w_p}=\derX_i(w_p) =-\sum_{s=1}^m a_{ip}^s w_s,\quad (n+1\leqslant p\leqslant m),
$$
$$
\BrAn{w_q}{w_p}=0,\quad (n+1\leqslant q,p\leqslant m).
$$
In the special case that $A_1,A_2,\ldots,A_n$ are constant matrices, Equation \eqref{Eqn:compatible condition 2 of Lie algebroid in special n system} reduces to
\begin{eqnarray*}\label{Eqt:constantAiAjcondition}
	\Br{A_i}{A_j}
	=(a_{ji}^k-a_{ij}^k)A_k.
\end{eqnarray*}
\subsection{DG manifolds from compatible $n$-systems}\label{DG} In this  subsection, we consider  differential graded (DG) manifolds derived from  compatible $n$-systems. Let $V$ be a $\mathbb{Z}$-graded vector space over $\mathbb{R}$, represented as $V = \oplus_{n\in \mathbb{Z}}V_n$, where each element $v\in V_n$ is assigned a degree $(+n)$, denoted as $|v|=n$. Suppose that $A$ is a Lie algebroid  arising from a compatible $n$-system $(\derX_1, \ldots, \derX_n; \derad^1, \ldots, \derad^n)$ underlying the vector bundle $A$. Given these data, we can construct a DG manifold $(A[1]=(M,\Gamma(\wedge^\bullet A^*)),D)$, where $D:\Gamma(\wedge^\bullet A^*) \to \Gamma(\wedge^{\bullet+1} A^*)$ is defined as the Chevalley-Eilenberg differential of the Lie algebroid.  In fact, this differential can be expressed as $D:=\derad^i\otimes \derX_i$, by some easy  verification.
In conclusion, $(A[1],D=\derad^i\widehat{\otimes} \derX_i)$ constitutes a DG manifold derived from the said compatible $n$-system.

\subsection{Construction of  Lie pairs from  compatible $n$-systems}\label{Lie pair}

Consider a compatible $n$-system $(\derX_1$, $\ldots$, $\derX_n$; $\derad^1, \ldots, \derad^n)$ defined on $A$. Suppose that  the covariant differential operators $(\derX_1, \ldots, \derX_n)$ within this system can be  divided into two subsets: $(\derX_1, \ldots, \derX_k)$ and $(\derX_{k+1}, \ldots, \derX_n)$, and the subset $(\derX_1, \ldots, \derX_k)$ exhibits closure under the commutator operation, and satisfies the following condition:
\begin{eqnarray}
	\label{Eqn:compatible condition of Lie pair in  n system}
	\BrCDO{\derX_{l_1}}{\derX_{l_2}}
	=\sum_{l_3=1}^k
	(a_{l_2l_1}^{l_3}-a_{l_1l_2}^{l_3})\derX_{l_3},
	\qquad\forall l_1,l_2=1,2,\ldots,k.
\end{eqnarray}
Here, the functions  $a_{ik}^j \in \CinfM$ are defined in Equation \eqref{Eqn:compatible condition 1 of Lie algebroid in n system}. The equation above corresponds to the following condition:
\begin{eqnarray*}
	a_{l_2l_1}^p
	=a_{l_1l_2}^p
	,
	\qquad\forall
	~l_1,l_2=1,2,\ldots,k; ~p=k+1,k+2,\ldots,n.
\end{eqnarray*}

We can define a subbundle $S$ of $A$ as
$$
S:=\{s\in A|
\lang{s}{\derad^{k+1}}=\lang{s}{\derad^{k+2}}=\cdots=\lang{s}{\derad^n}=0\}.
$$

Under these conditions, $(A,S)$  forms a Lie pair  as defined in \cite{CSX2}. Recall that a Lie pair is defined as follows-- A Lie pair $(A,S)$ is a pair consisting of two Lie algebroids 
over the same manifold $M$ together with a Lie algebroid inclusion $S\hookrightarrow A$ (covering the identity map $\id_M$).

 Resume the setting of $(A,S)$ as earlier. The structure maps of the  Lie subalgebroid $S$ are
\begin{eqnarray*}
	a_S(s)
	&:=\hspace*{-0.5em}&\anchorn(s)
	=\sum_{l=1}^k \lang{s}{\derad^l}\derx_l
	,
	\qquad\forall s\in S,
	\\
	~[s_1,s_2]_S
	&:=\hspace*{-0.5em}&\BrAn{s_1}{s_2}
	=\sum_{l=1}^k\Lang{s_1}{\derad^l}\cdot\derX_l(s_2)
	-\sum_{l=1}^k\Lang{s_2}{\derad^l}\cdot\derX_l(s_1)
	,
	\qquad
	\forall s_1,s_2\in S.
\end{eqnarray*}

The straightforward verification of these formulas is omitted for brevity. A simple method for constructing Lie pairs is to choose any covariant differential operator $\derX_i$. Using the method above, we can define a non-trivial Lie subalgebroid $S$ such that $(A,S)$ forms a Lie pair.

\section{From metric connections to Courant algebroids}\label{sec 3}
This section introduces compatible \textbf{metric} $n$-systems on \textbf{metric} bundles. We present   characterizations of metric algebroids, pre-Courant algebroids, and Courant algebroids through metric connections and metric $n$-systems.
\subsection{Metric connections and metric algebroids}\label{sec 3.1}
Consider a vector bundle $E \rightarrow M$ equipped with a pseudo-metric, which is a non-degenerate symmetric bilinear form denoted by $\lang{\tobefilledin}{\tobefilledin}$. In addition, we assume a bundle map $a_E: E \rightarrow TM$ is given, and designate it as the anchor map. The triple $(E, \lang{\tobefilledin}{\tobefilledin}, a_E)$ is called an \textbf{anchored metric bundle}. Given these data, a connection $\nabla:\secE\times\secE\rightarrow\secE$ on the anchored metric bundle $(E,\lang{\tobefilledin}{\tobefilledin},a_E)$ is called a  \textbf{metric connection} if it satisfies the compatibility condition:
\begin{eqnarray}
	\label{Eqn:compatible condition of pseudo-metric and metric connection}
	a_E(e_1)\lang{e_2}{e_3}=\Lang{\nabla_{e_1}e_2}{e_3}+\Lang{e_2}{\nabla_{e_1}e_3},\quad\forall e_1,e_2,e_3\in\secE.
\end{eqnarray}
The metric connection $\nabla$ is a special instance of the generalized connection described in \cite{CPR}.
To see the existence of such a connection, one takes an arbitrary $TM$-connection $\nabla^0:\secM\times\secE\rightarrow\secE$, and defines two auxiliary operators $H:\secM\otimes\secE\rightarrow\secE$ and $\nM:\secM\times\secE\rightarrow\secE$ as follows:
\begin{eqnarray*}
	&&\Lang{H(X,e_1)}{e_2}:=X\lang{e_1}{e_2}-\Lang{\nabla^0_{X}e_1}{e_2}-\Lang{e_1}{\nabla^0_{X}e_2},\\
	&&\nM_{X}e_1:=\nabla^0_{X}e_1+\frac{1}{2}H(X,e_1),
\end{eqnarray*}
where $X\in\secM$ and $e_1,e_2\in\secE$. One can easily check that $$H(X,fe_1)=H(fX,e_1)=fH(X,e_1),\quad \forall f\in C^\infty(M).$$ Then, the operator $\nM$ defines a $TM$-connection on $E$ and satisfies the following condition:
\begin{eqnarray*}
	X\lang{e_1}{e_2}=\Lang{\nM_{X}e_1}{e_2}+\Lang{e_1}{\nM_{X}e_2}.
\end{eqnarray*}

Subsequently, we define a metric connection $\nabla:\secE\times\secE\rightarrow\secE$ as the pullback of $\nM$:
\begin{eqnarray*}
	\nabla_{e_1}e_2:=\nM_{a_E(e_1)}e_2,\quad\forall e_1,e_2\in\secE.
\end{eqnarray*}

Then, this connection $\nabla$ satisfies the compatibility condition:
\begin{eqnarray*}
	\Lang{\nabla_{e_1}e_2}{e_3}+\Lang{e_2}{\nabla_{e_1}e_3}=\Lang{\nM_{a_E(e_1)}e_2}{e_3}+\Lang{e_2}{\nM_{a_E(e_1)}e_3}=a_E(e_1)\lang{e_2}{e_3}.
\end{eqnarray*}
\begin{definition}
	Given a metric connection $\nabla$ on an anchored metric bundle $\big(E, \lang{\tobefilledin}{\tobefilledin}, a_E\big)$, we define an associated bracket $\circnabla:~\Gamma(E)\times\Gamma(E)\rightarrow\Gamma(E)$  on $\Gamma(E)$ through the following relation: 
	\begin{align}\label{Def: Dorfman bracket of metric algebroid}\lang{e_1\circnabla e_2}{e_3}:=\lang{\nabla_{e_1} e_2}{e_3}-\lang{\nabla_{e_2} e_1}{e_3}+\lang{\nabla_{e_3} e_1}{e_2}. 
	\end{align}
 Furthermore, define an operator $\Delta:~\Gamma(E)\times\Gamma(E)\rightarrow\Gamma(E)$ by \begin{align}\label{Eqt:Delta} \lang{\Delta_{e_2}e_1}{e_3}:=\lang{\nabla_{e_3}e_1}{e_3},\qquad\forall e_1,e_2,e_3\in\Gamma(E). \end{align}
\end{definition}
We can easily verify  the following identities:
\begin{eqnarray*}
	\Delta_{(fe_2)}e_1=f\Delta_{e_2}e_1,\quad
	\Delta_{e_2}(fe_1)=f\Delta_{e_2}e_1+\lang{e_1}{e_2}\cdot \D(f),\quad
	\D\lang{e_1}{e_2}=\Delta_{e_1}e_2+\Delta_{e_2}e_1,
\end{eqnarray*}
where $\D: \CinfM \rightarrow \secE$ is given by $\lang{ \D(f)}{e} = a_E(e)(f)$ for all $e\in \Gamma(E)$ and  $f \in \CinfM$.

The associated  bracket  $\circnabla$ can also be expressed by
\begin{eqnarray}\label{Def: Dorfman bracket in brief of metric algebroid}
	e_1\circnabla e_2=\nabla_{e_1}e_2-\nabla_{e_2}e_1+\Delta_{e_2}e_1.
\end{eqnarray}
Consequently,  we obtain the following identities:
\begin{eqnarray*}
	a_E(e)\lang{h_1}{h_2}=\lang{e\circnabla h_1}{h_2}+\lang{h_1}{e\circnabla h_2},\qquad
	e\circnabla e=\frac{1}{2}\D\lang{e}{e}.
\end{eqnarray*}

Thus, the quadruple $\big(E,\lang{\tobefilledin}{\tobefilledin},\circnabla,a_E\big)$ forms a metric algebroid. 
\begin{prop}\label{prop:surj-Metric algebroid}
	Let  $\big(E,\lang{\tobefilledin}{\tobefilledin},\circ,a_E\big)$ be a metric algebroid. Then there exists a metric connection $\nabla$ on the anchored metric bundle $\big(E,\lang{\tobefilledin}{\tobefilledin},a_E\big)$ such that the operation $\circ$ is exactly the operation $\circnabla$, as defined in Equation \eqref{Def: Dorfman bracket of metric algebroid} or \eqref{Def: Dorfman bracket in brief of metric algebroid}.
\end{prop}
\begin{proof}
	Consider an arbitrary metric connection $\nabla^0: \secE \times \secE \rightarrow \secE$ on the metric vector bundle $E$.  Consequently, we have an operator $\overset{\nabla^0}{\circ}: \secE \times \secE \rightarrow \secE$ defined as follows:
	\begin{eqnarray*}
		\Lang{e_1 \overset{\nabla^0}{\circ} e_2}{e_3}
		:= \Lang{\nabla^0_{e_1} e_2 - \nabla^0_{e_2} e_1}{e_3}
		+ \Lang{\nabla^0_{e_3} e_1}{e_2},
	\end{eqnarray*}
	for all  $e_1, e_2, e_3 \in \secE$, and we have a metric algebroid $\big(E, \lang{\tobefilledin}{\tobefilledin}, \overset{\nabla^0}{\circ}, a_E\big)$. However, $\overset{\nabla^0}{\circ}$ is differs the given operator $\circ$. To derive the desired connection $\nabla$, we introduce an operator $K: \secE \otimes \secE \rightarrow \secE$, which is defined by
	\begin{eqnarray*}
		K(e_1, e_2)
		:= e_1 \circ e_2
		- e_1 \overset{\nabla^0}{\circ} e_2,
		\qquad \forall e_1, e_2 \in \secE.
	\end{eqnarray*}
	Through direct calculation, it can be verified that $K$ is skew-symmetric and $\CinfM$-bilinear, and  satisfies the following identity:
	\begin{eqnarray*}
		\Lang{K(e_1, e_2)}{e_3} + \Lang{e_2}{K(e_1, e_3)} = 0,
		\qquad \forall e_1, e_2, e_3.
	\end{eqnarray*}
	Utilizing these properties, we construct the desired connection $\nabla: \secE \times \secE \rightarrow \secE$ as
	\begin{eqnarray*}
		\nabla_{e_1} e_2
		:= \nabla^0_{e_1} e_2
		+ \frac{1}{3} K(e_1, e_2),
		\qquad \forall e_1, e_2 \in \secE.
	\end{eqnarray*}
	It can be demonstrated  that $\nabla$ is indeed a metric connection on $\big(E, \lang{\tobefilledin}{\tobefilledin}, a_E\big)$, satisfying the following equation:
	\begin{eqnarray*}
		\Lang{e_1 \circ e_2}{e_3}
		= \Lang{\nabla_{e_1} e_2 - \nabla_{e_2} e_1}{e_3}
		+ \Lang{\nabla_{e_3} e_1}{e_2},
		\qquad \forall e_1, e_2, e_3 \in \secE.
	\end{eqnarray*}
	This completes the proof.
\end{proof}

\begin{definition}
	Two metric connections $\nabla$ and $\nabla'$ are said to be equivalent, written as $\nabla\sim \nabla'$, if their difference is   a vector bundle map
$T:E\otimes E\rightarrow E$  satisfying the property
	$$\Lang{T(e_2,e_1)}{e_3}=\Lang{T(e_1,e_2)}{e_3}+\Lang{T(e_3,e_1)}{e_2},
\qquad \forall e_1,e_2,e_3\in\secE. $$
	
\end{definition}

Using Proposition \ref{prop:surj-Metric algebroid}, we can directly derive the next proposition.
\begin{prop}
There exists a one-to-one correspondence between   equivalence classes of metric connections on the anchored metric bundle $(E, \lang{\tobefilledin}{\tobefilledin}, a_E)$ and   metric algebroid structures underlying  $\big(E,\lang{\tobefilledin}{\tobefilledin}, a_E\big)$.
\end{prop}
\begin{definition}
	Assume that $\big(E$, $\lang{\tobefilledin}{\tobefilledin}$, $\circnabla$, $a_E\big)$
	is a metric algebroid arising from a metric connection $\nabla$.
Define  a bilinear map $\cnabla :\secE\times\secE\rightarrow\Gamma(\DE)$  by
	\begin{eqnarray}\label{Def: newcurvature-CA}
		\cnabla (e_1,e_2)
		&:=&
		\nabla_{e_1} \nabla_{e_2}
		-\nabla_{e_2} \nabla_{e_1}
		-\nabla_{(e_1\circnabla  e_2)}
		\nonumber\\&=&
		\nabla_{e_1} \nabla_{e_2}
		-\nabla_{e_2} \nabla_{e_1}
		-\nabla_{(\nabla_{e_1} e_2)}
		+\nabla_{(\nabla_{e_2} e_1)}
		-\nabla_{(\Delta_{e_2}e_1)},
	\end{eqnarray}
	for all $e_1,e_2\in\secE$.
	We  call $\cnabla$ the \textbf{Courant curvature } associated with the metric connection $\nabla$.
\end{definition}
\begin{Rem}
	Consider a pure algebroid $\big(A$, $\Brn{\tobefilledin}{\tobefilledin}$, $a_A\big)$ arising from a connection $\nabla$, and a trivial Lie algebroid $\big(A^*$, $\BrAstar{\tobefilledin}{\tobefilledin}=0$, $a_{A^*}=0\big)$. Using these structures, we can construct an anchored metric bundle $(E=A\oplus A^*,\lang{\tobefilledin}{\tobefilledin}, a_E=\rhoA)$, where $\lang{\tobefilledin}{\tobefilledin}$ is the standard pairing as defined in \eqref{Eqn:standard pairing}. This construction naturally induces a metric connection
	\[
	\widehat{\nabla}:\secE\times\secE\rightarrow\secE,\qquad \widehat{\nabla}_{x+\xi}(y+\eta):=\nabla_xy+\nabla_x\eta,
	\]
	for all  $x+\xi,y+\eta\in\secE$.
	
	The above structures give rise to a metric algebroid $\big(E,\lang{\tobefilledin}{\tobefilledin}, \overset{\widehat{\nabla}}{\circ},a_E=\rhoA\big)$, where the Dorfman bracket $\overset{\widehat{\nabla}}{\circ}$ can be explicitly expressed as
	\begin{eqnarray*}
		(x+\xi)\overset{\widehat{\nabla}}{\circ}(y+\eta)
		&=&\Brn{x}{y}
		+
		L_x\eta
		-\iota_y(d_A \xi),
	\end{eqnarray*}
	for all  $x+\xi,y+\eta\in\secE$. The associated Courant curvature $C^{\widehat{\nabla}}$ is of the form
	\begin{eqnarray*}
		C^{\widehat{\nabla}}(x+\xi,y+\eta)
		=\rnabla (x,y).
	\end{eqnarray*}
This formulation shows that Courant curvatures naturally extend    curvatures  of pure algebroids.
\end{Rem}
By direct computations, we can verify that the Courant curvature $\cnabla $ satisfies the following properties:
\begin{eqnarray}
\cnabla (e_1,e_2)
+\cnabla (e_2,e_1)
&\hspace*{-0.5em}=\hspace*{-0.5em}&-\nabla_{\D\lang{e_1}{e_2}},
\nonumber\\
\cnabla (fe_1,e_2)
&\hspace*{-0.5em}=\hspace*{-0.5em}&f\cnabla (e_1,e_2)
-\lang{e_1}{e_2}\cdot\nabla_{\D(f)}
,
\nonumber\\
\label{Eqn: Cnabla property-3}
\cnabla (e_1,fe_2)
&\hspace*{-0.5em}=\hspace*{-0.5em}&f\cnabla (e_1,e_2)
,
\\
\label{Eqn: Cnabla property-4}
\cnabla (e_1,e_2)(fe_3)
&\hspace*{-0.5em}=\hspace*{-0.5em}&f\cnabla (e_1,e_2)e_3
+\big(
\BrM{a_E(e_1)}{a_E(e_2)}-a_E(e_1\circnabla  e_2)
\big)(f)\cdot e_3,
\end{eqnarray}
for all $e_1,e_2,e_3\in\secE, f\in\CinfM$.

\begin{Rem}
The curvature $R_{\nabla}$ introduced in \cite{MJ2018} differs from our  Courant curvature $C^\nabla$ due to two reasons -- $(1)$ What she called the Dorfman connection is not our metric connection; and $(2)$, her  $R_{\nabla}$ is not linear in its second argument, i.e., $R_{\nabla}(e_1,fe_2)\neq fR_{\nabla}(e_1,e_2)$ for all $e_1,e_2\in\Gamma(E)$ and $f\in\CinfM$.
\end{Rem}
\begin{prop}\label{Prop:preCourantCondition}
The quadruple $\big(E,\lang{\tobefilledin}{\tobefilledin},\circnabla,a_E\big)$ is a pre-Courant algebroid if and only if   either of the following  equivalent conditions holds:
\begin{itemize}
	\item The operator $\cnabla(e_1,e_2):\secE\rightarrow\secE$ is  $\CinfM$-linear  for all  $e_1\in\secE$;
	\item For any   $e\in\secE$ and   $f\in\CinfM$, the relation $\D(f)\circnabla e=0$ holds.
\end{itemize}
\end{prop}
To establish the key theorem that fully characterizes Courant algebroids through the Courant curvature, we introduce an operator $Q:\secE\times\secE\times\secE\rightarrow\secE$ defined by
\begin{eqnarray}\label{Def:Q}
	\Lang{Q(e_1,e_2,e_3)}{t}
	:=\Lang{\cnabla (e_1,t)e_2}{e_3}
	+\Lang{\cnabla (e_2,t)e_3}{e_1}
	+\Lang{\cnabla (e_3,t)e_1}{e_2},
\end{eqnarray}
where $e_1,e_2,e_3,t\in\secE$. The $\CinfM$-linearity of this relation with respect to the $t$-argument, as derived from Equation \eqref{Eqn: Cnabla property-3}, ensures that $Q(e_1,e_2,e_3)$ is well-defined.
\begin{Thm}\label{courant algebroid curvature}Let $\big(E$, $\lang{\tobefilledin}{\tobefilledin}$, $\circnabla$, $a_E\big)$
	be a metric algebroid.   If $\cnabla$ is $\CinfM$-linear with respect to its third argument, and the following identity holds:
	\begin{eqnarray}\label{Eqn: pmBI of pre-Courant algebroid}
		&&
		\cnabla (e_1,e_2)e_3
		+\cnabla (e_2,e_3)e_1
		+\cnabla (e_3,e_1)e_2
		\nonumber\\&&\quad
		+\nabla_{(\Delta_{e_3}e_2)}e_1
		+\nabla_{(\Delta_{e_1}e_3)}e_2
		+\nabla_{(\Delta_{e_2}e_1)}e_3
		\nonumber\\&&\quad\qquad
		+Q(e_1,e_2,e_3)
		=0,
	\end{eqnarray}
	for all $e_1,e_2,e_3\in\secE$,
	then $\big(E$, $\lang{\tobefilledin}{\tobefilledin}$, $\circnabla$, $a_E\big)$
	forms a Courant algebroid.
\end{Thm}
 We call Equation \eqref{Eqn: pmBI of pre-Courant algebroid} the \textbf{Courant-Bianchi identity} (of the Courant algebroid  $(E$, $\lang{\tobefilledin}{\tobefilledin}$, $\circnabla$, $a_E))$.

The proof of Theorem \ref{courant algebroid curvature}  will follow immediately
from Equation \eqref{Eqn: Cnabla property-4} and the proposition below. 
\begin{prop}
The following equality holds:
\begin{eqnarray}\label{Eqn: mBI of metric algebroid}
&&
\cnabla (e_1,e_2)e_3
+\cnabla (e_2,e_3)e_1
+\cnabla (e_3,e_1)e_2 +\nabla_{(\Delta_{e_3}e_2)}e_1
\nonumber\\&&\quad
+\nabla_{(\Delta_{e_1}e_3)}e_2
+\nabla_{(\Delta_{e_2}e_1)}e_3+Q(e_1,e_2,e_3)
+(\D\lang{e_1}{e_3})\circnabla  e_2
\nonumber\\&\hspace*{-0.5em}=\hspace*{-0.5em}&
e_1\circnabla (e_2\circnabla  e_3)
-(e_1\circnabla  e_2)\circnabla  e_3
-e_2\circnabla (e_1\circnabla  e_3),
\end{eqnarray}
for all $e_1,e_2,e_3\in\secE$.
\end{prop}
We call Equation \eqref{Eqn: mBI of metric algebroid} the \textbf{metric-Bianchi identity} $\big($of the metric algebroid $\big(E$, $\lang{\tobefilledin}{\tobefilledin}$, $\circnabla$, $a_E\big)$$\big)$.

{\begin{proof}

For all $e_1,e_2,e_3,t\in\secE$, by Equations \eqref{Eqn:compatible condition of pseudo-metric and metric connection} and \eqref{Def: Dorfman bracket of metric algebroid},
we have
\begin{eqnarray*}
\lang{e_1\circnabla (e_2\circnabla  e_3)}{t}
&\hspace*{-0.5em}=\hspace*{-0.5em}&\lang{\nabla_{e_1}\nabla_{e_2}e_3}{t}
-\lang{\nabla_{e_1}\nabla_{e_3}e_2}{t}
-\lang{\nabla_{(e_2\circnabla  e_3)}e_1}{t}
\\&&\quad
+\lang{\nabla_te_1}{\nabla_{e_2}e_3}
-\lang{\nabla_te_1}{\nabla_{e_3}e_2}
+\lang{\nabla_te_2}{\nabla_{e_1}e_3}
\\&&\quad\qquad
+\lang{\nabla_{e_1}\nabla_te_2}{e_3}
-\lang{\nabla_{(\nabla_{e_1}t)}e_2}{e_3}
+\lang{\nabla_{(\nabla_te_1)}e_2}{e_3}.
\end{eqnarray*}
By rotation of arguments, we   derive the following identity:
\begin{eqnarray}\label{Eqn: mBI-technology-1}
e_1\circnabla (e_2\circnabla  e_3)+c.p.(e_1,e_2,e_3)
&\hspace*{-0.5em}=\hspace*{-0.5em}&\cnabla (e_1,e_2)e_3
+\nabla_{(\Delta_{e_3}e_2)}e_1
+\D\lang{e_1}{\nabla_{e_2}e_3}
+c.p.(e_1,e_2,e_3)
\nonumber\\&&\quad
+Q(e_1,e_2,e_3),
\end{eqnarray}
where $c.p.$ denotes the cyclic permutation of the entries.
Using $e_1\circnabla  e_2+e_2\circnabla  e_1=\D\lang{e_1}{e_2}$,  we have two identities:
\begin{eqnarray}
\label{Eqn: mBI-technology-2}
-(e_1\circnabla  e_2)\circnabla  e_3
&\hspace*{-0.5em}=\hspace*{-0.5em}&
e_3\circnabla (e_1\circnabla  e_2)
-\D\lang{e_1\circnabla  e_2}{e_3}
,
\\
\label{Eqn: mBI-technology-3}
-e_2\circnabla (e_1\circnabla  e_3)
&\hspace*{-0.5em}=\hspace*{-0.5em}&
e_2\circnabla (e_3\circnabla  e_1)
-e_2\circnabla \big(\D\lang{e_1}{e_3}\big).
\end{eqnarray}
We also need  the following relation:
\begin{eqnarray}\label{Eqn: mBI-technology-4}
a_E(e_2\circnabla  t)\lang{e_1}{e_3}
&\hspace*{-0.5em}=\hspace*{-0.5em}&
\Lang{e_2\circnabla  t}{\D\lang{e_1}{e_3}}
\nonumber\\
&\hspace*{-0.5em}=\hspace*{-0.5em}&
\Lang{\nabla_{e_2}t-\nabla_te_2}{\D\lang{e_1}{e_3}}
+\Lang{\nabla_{\D\lang{e_1}{e_3}}e_2}{t}
\nonumber\\
&\hspace*{-0.5em}=\hspace*{-0.5em}&
a_E(\nabla_{e_2}t)\lang{e_1}{e_3}
-a_E(\nabla_te_2)\lang{e_1}{e_3}
+\Lang{\nabla_{\D\lang{e_1}{e_3}}e_2}{t}.
\end{eqnarray}

Then, we can derive the following equality: 
\begin{eqnarray}\label{Eqn: mBI-technology-5}
&&\Lang{-\D\lang{e_1\circnabla  e_2}{e_3}
-e_2\circnabla \big(\D\lang{e_1}{e_3}\big)}{t}
\nonumber\\&\hspace*{-0.5em}=\hspace*{-0.5em}&
-a_E(t)\big(
\lang{\nabla_{e_1}e_2}{e_3}
-\lang{\nabla_{e_2}e_1}{e_3}
+\lang{\nabla_{e_3}e_1}{e_2}
\big)
\nonumber\\&&\quad
-\lang{\nabla_{e_2}(\D\lang{e_1}{e_3})}{t}
+\lang{\nabla_{\D\lang{e_1}{e_3}}e_2}{t}
-\lang{\nabla_te_2}{\D\lang{e_1}{e_3}}
\nonumber\\&\hspace*{-0.5em}=\hspace*{-0.5em}&
-a_E(t)\big(
\lang{\nabla_{e_1}e_2}{e_3}
+\lang{\nabla_{e_2}e_3}{e_1}
+\lang{\nabla_{e_3}e_1}{e_2}
\big)+a_E(\nabla_{e_2}t)\lang{e_1}{e_3}
\nonumber\\&&\quad
-a_E(\nabla_te_2)\lang{e_1}{e_3}
+\lang{\nabla_{\D\lang{e_1}{e_3}}e_2}{t}
-a_E(e_2)(a_E(t)\lang{e_1}{e_3})
+a_E(t)(a_E(e_2)\lang{e_1}{e_3})
\nonumber\\
&\equalbyreason{\eqref{Eqn: mBI-technology-4}}&
-a_E(t)\big(
\lang{e_1}{\nabla_{e_2}e_3}
+\lang{e_2}{\nabla_{e_3}e_1}
+\lang{e_3}{\nabla_{e_1}e_2}
\big)
+\big(
a_E(e_2\circnabla  t)
-\BrM{a_E(e_2)}{a_E(t)}
\big)\lang{e_1}{e_3}.
\end{eqnarray}
Thus, we have
\begin{eqnarray*}
&&\Lang{e_1\circnabla (e_2\circnabla  e_3)
-(e_1\circnabla  e_2)\circnabla  e_3
-e_2\circnabla (e_1\circnabla  e_3)}{t}
\\&\equalbyreason{\eqref{Eqn: mBI-technology-2}\eqref{Eqn: mBI-technology-3}\eqref{Eqn: mBI-technology-5}}&
\Lang{e_1\circnabla (e_2\circnabla  e_3)}{t}
+\Lang{e_2\circnabla (e_3\circnabla  e_1)}{t}
+\Lang{e_3\circnabla (e_1\circnabla  e_2)}{t}
\\&&\quad
-a_E(t)\big(
\lang{e_1}{\nabla_{e_2}e_3}
+\lang{e_2}{\nabla_{e_3}e_1}
+\lang{e_3}{\nabla_{e_1}e_2}
\big)
+\big(
a_E(e_2\circnabla  t)
-\BrM{a_E(e_2)}{a_E(t)}
\big)\lang{e_1}{e_3}
\\&\equalbyreason{\eqref{Eqn: mBI-technology-1}\eqref{Eqn:Metric algebroid D and anchor condition}}&\Lang{
\cnabla (e_1,e_2)e_3
+\cnabla (e_2,e_3)e_1
+\cnabla (e_3,e_1)e_2+\nabla_{(\Delta_{e_3}e_2)}e_1
\\&&\quad
+\nabla_{(\Delta_{e_1}e_3)}e_2
+\nabla_{(\Delta_{e_2}e_1)}e_3
+Q(e_1,e_2,e_3)
+\big(\D\lang{e_1}{e_3}\big)\circnabla  e_2
}{t}
.
\end{eqnarray*}
This completes the proof.
\end{proof}
}
\subsection{Three typical examples of metric algebroids}\label{three typical examples}
We will examine three types of metric bundles:  $E=\AAstar$ (where $A$ is an arbitrary pure algebroid), $\MMstar$ (where $M$ is an arbitrary manifold), and $\MM$, and establish the corresponding metric, pre-Courant, and Courant algebroid structures.
\vskip 0.1cm
\textbf{Case 1:} $E=\AAstar$.\label{subsubsectionAAstar}

 Let $(A,\BrA{\tobefilledin}{\tobefilledin},\rhoA)$ and
$(A^*,\BrAstar{\tobefilledin}{\tobefilledin},a_{A^*})$ be two pure algebroids arising from two  torsion-free connections $\nnAA:\secA\times\secA\rightarrow\secA$ and
$\nnAstarAstar:\secAstar\times\secAstar\rightarrow\secAstar$, respectively.
The standard metric on $A \oplus A^*$ is defined by the standard pairing \eqref{Eqn:standard pairing}. The anchor map $a_E: \AAstar\rightarrow TM$ is defined  by
$$a_E(x+\xi)=\rhoA(x)+a_{A^*}(\xi),$$
for all $x+\xi\in\secAAstar$.

These structures  establish an anchored metric bundle $(\AAstar,\lang{\tobefilledin}{\tobefilledin},a_E)$.
By $\nnAA$ and $\nnAstarAstar$, we obtain two induced connections
$\nnAAstar:\secA\times\secAstar\rightarrow\secAstar$ and
$\nnAstarA:\secAstar\times\secA\rightarrow\secA$ given respectively by
\begin{eqnarray*}
\Lang{\nnAAstar_x\xi}{y}
&\hspace*{-0.5em}=\hspace*{-0.5em}&\rhoA(x)\lang{\xi}{y}
-\Lang{\xi}{\nnAA_xy}
,
\\\mbox{and }\quad 
\Lang{\nnAstarA_\xi x}{\eta}
&\hspace*{-0.5em}=\hspace*{-0.5em}&a_{A^*}(\xi)\lang{x}{\eta}
-\Lang{x}{\nnAstarAstar_\xi\eta},
\qquad\forall x,y\in\secA,\xi,\eta\in\secAstar.
\end{eqnarray*}
Thus, we have a natural operator $\nabla:\secAAstar\times\secAAstar\rightarrow\secAAstar$ defined by
\begin{eqnarray*}
\nabla_{x+\xi}(y+\eta)
:=\nnAA_xy+\nnAAstar_x\eta+\nnAstarA_\xi y+\nnAstarAstar_\xi\eta,
\end{eqnarray*}
and the associated Dorfman bracket $\circnabla:\secAAstar\times\secAAstar\rightarrow\secAAstar$ given by
\begin{eqnarray}\label{Def: Dorfman bracket of AAstar}
(x+\xi)\circnabla (y+\eta)
:=\nabla_{x+\xi}(y+\eta)-\nabla_{y+\eta}(x+\xi)+\Delta_{y+\eta}(x+\xi),
\end{eqnarray}
where $\Delta:\secAAstar\times\secAAstar\rightarrow\secAAstar$ (according to Equation \eqref{Eqt:Delta}) is explicitly expressed by
\begin{eqnarray*}
\Lang{\Delta_{y+\eta}(x+\xi)}{z+\gamma}
=\Lang{\nabla_{z+\gamma}(x+\xi)}{y+\eta}
=\Lang{\nnAA_zx+\nnAstarA_\gamma x}{\eta}
+\Lang{\nnAAstar_z\xi+\nnAstarAstar_\gamma\xi}{y},
 \end{eqnarray*}
for all $ x+\xi,y+\eta,z+\gamma\in\secAAstar$. 

Further,  the operator $\circnabla$ given by Equation \eqref{Def: Dorfman bracket of AAstar} can be expressed alternatively:
\begin{eqnarray}\label{Def: Dorfman bracket of AAstar in detail}
(x+\xi)\circnabla (y+\eta)
=\big(\BrA{x}{y}
+L_\xi y
-\iota_\eta (\dAstar x)\big)
+\big(\BrAstar{\xi}{\eta}
+L_x\eta
-\iota_y(\dA \xi)\big).
\end{eqnarray}
In conclusion, $(\AAstar,\lang{\tobefilledin}{\tobefilledin},\circnabla,a_E)$  is a   metric algebroid.  One may ask when $(\AAstar,\lang{\tobefilledin}{\tobefilledin},\circnabla,a_E)$ is a pre-Courant  algebroid. The answer is a system of equations that we do not spell out here.
Moreover, if $(A,A^*)$ is a Lie bialgebroid in the sense of \cite{MX}, then $\AAstar$ is a Courant algebroid and  the canonical Dorfman bracket of $\AAstar$ is exactly defined as above (see  \cites{LWX}).

\textbf{Case 2: $E=\MMstar$. } 

Let $A=TM$ be the tangent Lie algebroid and $A^*=T^*M$ equipped with the trivial Lie algebroid structure. The anchor map $a_E: \secMMstar \rightarrow \secM$ is defined by:
\begin{eqnarray*}
	a_E (X+\alpha)=X, \qquad\forall X+\alpha\in\secMMstar.
\end{eqnarray*}
Moreover, the  Dorfman bracket \eqref{Def: Dorfman bracket of AAstar in detail} transforms into the following expression:
\begin{eqnarray}\label{Def: Dorfman bracket of MMstar}
(X+\alpha)\circnabla(Y+\beta)
=
\BrM{X}{Y}
+L_{X}\beta
-\iota_{Y}\dd\alpha
, \end{eqnarray}
where $\dd:\Gamma(\wedge^\bullet T^*M)\rightarrow\Gamma(\wedge^{\bullet+1}T^*M)$ is the de Rham differential. Indeed, the quadruple $\big(\MMstar$,  $\lang{\tobefilledin}{\tobefilledin}$, $\circnabla$, $a_E\big)$ is the first example of   Courant algebroids \cite{Courant}.

The standard Courant bracket $\circnabla$ in Equation \eqref{Def: Dorfman bracket of MMstar} can be altered through a twisting operation and a closed differential 3-form $\omega \in \Omega^3(M) = \Gamma(\wedge^3 T^*M)$. The modified Courant bracket, denoted $\overset{\omega}{\circ}$, is defined as
\begin{eqnarray}\label{Def: Dorfman bracket of MMstar twisted by omega}
	(X+\alpha)\overset{\omega}{\circ}(Y+\beta)
	:=
	(X+\alpha)\circnabla(Y+\beta)
	+\omega(X,Y,\tobefilledin).
\end{eqnarray}
Indeed, we can rewrite $\overset{\omega}{\circ}$ via a particular type of metric connections on $(\MMstar, \lang{\tobefilledin}{\tobefilledin},a_E )$  as follows:
\begin{enumerate}
	\item Consider an arbitrary torsion-free connection $\nabla^{TM}: \Gamma(TM) \times \Gamma(TM) \rightarrow \Gamma(TM)$. This connection induces a corresponding connection $\nabla^{T^*M}: \Gamma(TM) \times \Gamma(T^*M) \rightarrow \Gamma(T^*M)$ on the cotangent bundle $T^*M$, defined by the following equation: $$\lang{\nabla^{T^*M}_X \alpha}{Y}=X\lang{\alpha}{Y} -\lang{\alpha}{\nabla^{TM}_X Y }, \quad \forall X, Y \in \Gamma(TM), \alpha \in \Gamma(T^*M).$$
	
	\item Define an associated connection $\nabla:\secMMstar\times\secMMstar\rightarrow\secMMstar$ by 
	\begin{eqnarray}\label{Def: connection of MMstar}
		\nabla_{X+\alpha}(Y+\beta)
		=
		\nM_XY
		+\nMstar_X\beta
		,\qquad\forall X+\alpha,Y+\beta\in\secMMstar.
	\end{eqnarray}
	Then  $\nabla$ is a metric connection on $TM\oplus  T^*M$.

	\item Define an $\omega$-twisted connection $\nabla^{\omega}:\secMMstar\times\secMMstar\rightarrow\secMMstar$ by
	\begin{eqnarray*}\label{Eqn:8-8-1}
	\nabla^{\omega}_{X+\alpha}(Y+\beta)
	= \nabla_{X+\alpha}(Y+\beta)+\frac{1}{3}\omega(X,Y)
	,\qquad\forall X+\alpha,Y+\beta\in\secMMstar,
	\end{eqnarray*}
	where $\nabla$ is the metric connection defined in Equation  \eqref{Def: connection of MMstar}.
	Then  $\nabla^{\omega}$ is also a metric connection on $\secMMstar$.
	\item It can be directly verified that the Courant bracket $\circnabla$ from Equation \eqref{Def: Dorfman bracket of MMstar} is induced by the connection $\nabla$.  Similarly, the $\omega$-twisted Courant bracket $\overset{\omega}{\circ}$ defined in Equation \eqref{Def: Dorfman bracket of MMstar twisted by omega} is induced by the modified connection $\nabla^{\omega}$.
\end{enumerate}
\textbf{Case 3: $E=\MM$. } 

Let $M$ be a smooth manifold and $g:TM\otimes TM\rightarrow \R$ a  pseudo-Riemannian metric. Thus, one can define an induced pseudo-metric
$\lang{\tobefilledin}{\tobefilledin}:\big(\MM\big)\otimes \big(\MM\big)\rightarrow\mathbb{R}$
by
\begin{eqnarray*}
	\langg{(X_1,X_2)}{(Y_1,Y_2)}
	:=\langr{X_1}{Y_2}
	+\langr{X_2}{Y_1},
	\quad\forall (X_1,X_2),(Y_1,Y_2)\in\secMM.
\end{eqnarray*}
Define the anchor map $\rho:\secMM\rightarrow\secM$ by
$\rho (X_1,X_2):=X_1+X_2$, for all $(X_1,X_2)\in\secMM$.
So, the triple $(\MM$, $\lang{\tobefilledin}{\tobefilledin}$, $\rho )$ is an anchored metric bundle.

Since we have an isomorphism of vector bundles   $g^\sharp:TM\rightarrow T^*M$ given by
$
\Lang{g^\sharp(X)}{Y}
:=g(X,Y)$,
for all $X,Y\in\secM$, one can identify $T^*M$ with $TM$ and endow $T^*M$ with a Lie algebroid structure isomorphic to $TM$.  Then consider the particular instance of $A=TM$ in \textbf{Case 1}, and we see that $\MM$   $\cong \MMstar$ has a metric algebroid structure. However,   \textbf{it can never become a pre-Courant algebroid}. The reason is as follows.

We can first take the  Levi-Civita connection $\nM:\secM\times\secM\rightarrow\secM$
arising from $g$; hence $\nM$ is torsion-free.  Then define an associated operator
$\DeltaTM:\secM\times\secM\rightarrow\secM$ by
$$
\Langr{\DeltaTM_{w}u}{v}
:=\Langr{\nM_{v}u}{w}
,\quad
\forall u,v,w\in\secM . $$

Meanwhile, we have the associated metric connection in $\MM$, i.e., $\nabla:\secMM\times\secMM\rightarrow\secMM$, given by
\begin{eqnarray*}
	\nabla_{(X_1,X_2)}(Y_1,Y_2)
	:=\big(\nM_{X_1+X_2}Y_1,
	\nM_{X_1+X_2}Y_2\big),
	\qquad\forall (X_1,X_2),(Y_1,Y_2)\in\secMM.
\end{eqnarray*}

As for the Dorfman bracket  $\circnabla :~\secMM\times \secMM\rightarrow\secMM
$
which is generally expressed by
\begin{eqnarray}
	(X_1,X_2)\circnabla  (Y_1,Y_2)
	:=\nabla_{(X_1,X_2)} (Y_1,Y_2)
	-\nabla_{(Y_1,Y_2)} (X_1,X_2)
	+\Delta_{(Y_1,Y_2)}(X_1,X_2),
	\nonumber
\end{eqnarray}
for all $(X_1,X_2),(Y_1,Y_2)\in\secMM$, we can compute the following four special situations:
\begin{eqnarray*}
	(X_1,0)\circnabla  (Y_1,0)
	&=&
	\big(\BrM{X_1}{Y_1},0\big)
	,\\
	(X_1,0)\circnabla  (0,Y_2)
	&=&
	\big(
	-\nM_{Y_2}X_1+\DeltaTM_{Y_2}X_1,
	\nM_{X_1}Y_2+\DeltaTM_{Y_2}X_1
	\big)
	,\\
	(0,X_2)\circnabla  (Y_1,0)
	&=&
	\big(
	\nM_{X_2}Y_1+\DeltaTM_{Y_1}X_2,
	-\nM_{Y_1}X_2+\DeltaTM_{Y_1}X_2
	\big)
	,\\
	(0,X_2)\circnabla  (0,Y_2)
	&=&
	\big(0,\BrM{X_2}{Y_2}\big).
\end{eqnarray*}

In order that the said metric algebroid $\big(\MM$,
$\langg{\tobefilledin}{\tobefilledin}$, $\circnabla$, $\rho\big)$
is a pre-Courant algebroid, we need
$$
\rho((X_1,0)\circnabla (0,Y_2))=\BrM{\rho(X_1,0)}{\rho(0,Y_2)}
$$
which is just the condition $\DeltaTM_{Y_2}X_1=0$, $\forall Y_2,X_1\in \secM$, and hence we must have  $\nM=0$, a contradiction.
\subsection{Construction of  Courant algebroids from  compatible metric $n$-systems}\label{sec 3.3} In this  section, we state our
second main result.  Consider a vector bundle \(E \rightarrow M\) equipped with a pseudo-metric $\lang{\cdotp}{\cdotp}$ on \(E\) and suppose that $\dere^1, \dere^2, \ldots, \dere^n$ are elements of $\Gamma(E)$ that are linearly independent over \(C^\infty(M)\). Furthermore, let \(\{\derZ_1, \derZ_2, \ldots, \derZ_n\} \subset \Gamma(\mathrm{CDO}(E))\) be covariant differential operators on $E$ (which may be linearly dependent over \(C^\infty(M)\)). We consistently represent the symbols of \(\{\derZ_1, \ldots, \derZ_{n}\}\) by the letters \(\{\derz_1, \ldots, \derz_{n}\}\) (where $\derz_i\in \Gamma(TM)$ for each $i$). Assume that the operators \(\derZ_i\), their corresponding symbols \(\derz_i\), and the pseudo-metric $\lang{\cdotp}{\cdotp}$ satisfy the following compatibility condition:
\begin{eqnarray}\label{Eqn: compatible condition of pseudo-metric and CDO}
	\Lang{\derZ_i(e_1)}{e_2}
	+\Lang{e_1}{\derZ_i(e_2)}
	=\derz_i\lang{e_1}{e_2},
	\qquad\forall e_1,e_2\in\secE.
\end{eqnarray}
In this part, we suppose that the rank of the vector bundle $E$ is at least $n$.
\begin{definition}\label{metric n-system on E} 	The datum  $$(\derZ_1,\ldots,\derZ_n;\dere^1,\ldots,\dere^n)$$ described as above  is  called a \textbf{metric $n$-system}.
\end{definition}

Given a metric $n$-system, one can  define the following structure maps on $\Gamma(E)$, which constitute a metric algebroid $\big(E$, $\lang{\tobefilledin}{\tobefilledin}$, $\circnabla$, $a_E\big)$.
\begin{itemize}
\item
The anchor map $a_E: E \rightarrow TM$ is defined by the expression:
\[
a_E(e) := \Lang{e}{\dere^i} \cdot \derz_i, \quad \forall e \in \secE.
\]
The associated operator $\D: \CinfM \rightarrow \secE$ is given by the following formula $$\D(f) := \derz_i(f) \cdot \dere^i,\hspace*{2em}\forall f \in \CinfM.$$
\item
The associated connection $\nabla:\secE\times\secE\rightarrow\secE$ is defined by the expression:
\begin{eqnarray}\label{Def: connection of Courant algebroid in n system}
\nabla_{e_1}e_2
:=\Lang{e_1}{\dere^i}\cdot\derZ_i(e_2),
\qquad\forall e_1,e_2\in\secE.
\end{eqnarray}
 The connection $\nabla$ is a metric connection because it satisfies Equation \eqref{Eqn: compatible condition of pseudo-metric and CDO}.
The associated operator $\Delta:\secE\times\secE\rightarrow\secE$ corresponding to $\nabla$ is given by the following formula
$$\Delta_{e_2}e_1:=\Lang{\derZ_i(e_1)}{e_2}\cdot\dere^i,\hspace*{2em}\forall e_1,e_2\in\secE.$$
\item
The Dorfman bracket
$\circnabla :\secE\times\secE\rightarrow\secE$ induced by $\nabla$ can be expressed by
\begin{eqnarray}
\label{Def: Dorfman bracket of Courant algebroid in n system}
e_1\circnabla e_2
&:=\hspace*{-0.5em}&\nabla_{e_1}e_2
-\nabla_{e_2}e_1
+\Delta_{e_2}e_1
\nonumber\\&=\hspace*{-0.5em}&
\Lang{e_1}{\dere^i}\cdot\derZ_i(e_2)
-\Lang{e_2}{\dere^i}\cdot\derZ_i(e_1)
+\Lang{\derZ_i(e_1)}{e_2}\cdot \dere^i
,\qquad\forall e_1,e_2\in\secE
.
\end{eqnarray}
\end{itemize}
The next step is to investigate the sufficient conditions under which this metric algebroid can be characterized as either
a pre-Courant algebroid or a Courant algebroid.

\begin{prop}\label{pre-Courant algebroid}
Suppose that the metric $n$-system $(\derZ_1,\ldots,\derZ_n;\dere^1,\ldots,\dere^n)$ on $E$ satisfies the following conditions:
\begin{eqnarray}
\label{Eqn: compatible condition 1 of Courant algebroid in n system}
\derZ_i(\dere^j)
&\hspace*{-0.5em}=\hspace*{-0.5em}&C_{ik}^j \dere^k,
\qquad\qquad\text{where}\quad C_{ik}^j\in\CinfM
,
\\
\label{Eqn: compatible condition 2 of pre-Courant algebroid in n system}
\BrCDO{\derz_i}{\derz_j}
&\hspace*{-0.5em}=\hspace*{-0.5em}&
(C_{ji}^k-C_{ij}^k)\derz_k
,
\\
\label{Eqn: compatible condition 3 of Courant algebroid in n system}
\Lang{ \dere^i}{\dere^j}
&\hspace*{-0.5em}=\hspace*{-0.5em}&0,
\end{eqnarray}
for all $i,j\in\{1,2,\ldots,n\}$. Then  the associated metric algebroid $\big(E$, $\lang{\tobefilledin}{\tobefilledin}$, $\circnabla$, $a_E\big)$ is a pre-Courant algebroid.
\end{prop}
\begin{proof}
We only need to verify the relation
\begin{eqnarray*}
a_E(e_1\circnabla e_2)-\Br{a_E(e_1)}{a_E(e_2)}=0,
\qquad\forall e_1,e_2\in\secE.
\end{eqnarray*}
In fact,   we have
\begin{eqnarray*}
&&a_E(e_1\circnabla e_2)-\Br{a_E(e_1)}{a_E(e_2)}
\\&\equalbyreason{\eqref{Def: Dorfman bracket of Courant algebroid in n system}}&
\Lang{e_1}{\dere^i}\cdot\Lang{\derZ_i(e_2)}{\dere^j}\cdot\derz_j
-\Lang{e_2}{\dere^i}\cdot\Lang{\derZ_i(e_1)}{\dere^j}\cdot\derz_j
+\Lang{\derZ_i(e_1)}{e_2}\cdot \lang{\dere^i}{\dere^j}\cdot\derz_j
\\&&\quad
-\lang{e_1}{\dere^i}\cdot\lang{e_2}{\dere^j}\cdot\Br{\derz_i}{\derz_j}
-\lang{e_1}{\dere^i}\cdot\derz_i\lang{e_2}{\dere^i}\cdot\derz_j
+\lang{e_2}{\dere^i}\cdot\derz_j\lang{e_1}{\dere^i}\cdot\derz_i
\\&\equalbyreason{\eqref{Eqn: compatible condition of pseudo-metric and CDO}}&
-\Lang{e_1}{\dere^i}\cdot\Lang{e_2}{\derZ_i(\dere^k)}\cdot\derz_k
+\Lang{e_2}{\dere^j}\cdot\Lang{e_1}{\derZ_j(\dere^k)}\cdot\derz_k
\\&&\quad
-\lang{e_1}{\dere^i}\cdot\lang{e_2}{\dere^j}\cdot\Br{\derz_i}{\derz_j}
+\Lang{\derZ_i(e_1)}{e_2}\cdot \lang{\dere^i}{\dere^j}\cdot\derz_j
\\&\equalbyreason{\eqref{Eqn: compatible condition 1 of Courant algebroid in n system}\eqref{Eqn: compatible condition 2 of pre-Courant algebroid in n system}}&
-\Lang{e_1}{\dere^i}\cdot\Lang{e_2}{C_{ij}^k\dere^j}\cdot\derz_k
+\Lang{e_2}{\dere^j}\cdot\Lang{e_1}{C_{ji}^k\dere^i}\cdot\derz_k
\\&&\quad
-\lang{e_1}{\dere^i}\cdot\lang{e_2}{\dere^j}\cdot(C_{ji}^k-C_{ij}^k)\derz_k
+\Lang{\derZ_i(e_1)}{e_2}\cdot \lang{\dere^i}{\dere^j}\cdot\derz_j
\\&\equalbyreason{\eqref{Eqn: compatible condition 3 of Courant algebroid in n system}}&
0.
\end{eqnarray*}
\end{proof}

\begin{prop}\label{Courant algebroid}
Consider the metric $n$-system $(\derZ_1,\ldots,\derZ_n;\dere^1,\ldots,\dere^n)$ on $E$, and suppose that it satisfies   Conditions \eqref{Eqn: compatible condition 1 of Courant algebroid in n system} and \eqref{Eqn: compatible condition 3 of Courant algebroid in n system} of the previous proposition, and  the following condition:
\begin{eqnarray}
\label{Eqn: compatible condition 2 of Courant algebroid in n system}
\BrCDO{\derZ_i}{\derZ_j}
&=&
(C_{ji}^k-C_{ij}^k)\derZ_k
, \end{eqnarray}
for all $i,j\in\{1,2,\ldots,n\}$.
Under these conditions, the following equation holds:
$$\cnabla (e_1,e_2)e_3
=\nabla_{(\Delta_{e_3}e_2)}e_1
=Q(e_1,e_2,e_3)
=(\D\lang{e_1}{e_3})\circnabla  e_2
=0
,\qquad\forall e_1,e_2,e_3\in\secE. $$
Consequently, the quadruple $\big(E$, $\lang{\tobefilledin}{\tobefilledin}$, $\circnabla$, $a_E\big)$ constitutes a Courant algebroid.
\end{prop}
 {
\begin{proof}
	Now, we verify that $\big(E$, $\lang{\tobefilledin}{\tobefilledin}$, $\circnabla$, $a_E\big)$ is a Courant algebroid.
	\begin{enumerate}
		\item[$\bullet$]To see $\cnabla (e_1,e_2)e_3$=0, we use Equations \eqref{Eqn: compatible condition 1 of Courant algebroid in n system} and \eqref{Eqn: compatible condition 2 of Courant algebroid in n system} to get
		\begin{eqnarray*}
			\cnabla (e_1,e_2)e_3
			&=&
			\nabla_{e_1} \nabla_{e_2}e_3
			-\nabla_{e_2} \nabla_{e_1}e_3
			-\nabla_{(e_1\circnabla  e_2)}e_3
			\\&=&\lang{e_1}{\dere^i}\cdot\lang{e_2}{\dere^j}\cdot\BrCDO{\derZ_i}{\derZ_j}(e_3)
			\\&&\qquad
			+\lang{e_1}{\dere^i}\cdot\lang{e_2}{\derZ_i(\dere^k)}\cdot\derZ_k(e_3)
			-\lang{e_2}{\dere^j}\cdot\lang{e_1}{\derZ_j(\dere^k)}\cdot\derZ_k(e_3)
			\\&&\qquad\qquad
			-\lang{\derZ_i(e_1)}{e_2}\cdot\lang{\dere^i}{\dere^k}\cdot\derZ_k(e_3)
			\\&=&\lang{e_1}{\dere^i}\cdot\lang{e_2}{\dere^j}\cdot(C_{ji}^k-C_{ij}^k)\derZ_k(e_3)
			\\&&\qquad
			+\lang{e_1}{\dere^i}\cdot\lang{e_2}{C_{ij}^k\dere^j}\cdot\derZ_k(e_3)
			-\lang{e_2}{\dere^j}\cdot\lang{e_1}{C_{ji}^k\dere^i}\cdot\derZ_k(e_3)
			\\&&\qquad\qquad
			-\lang{\derZ_i(e_1)}{e_2}\cdot\lang{\dere^i}{\dere^k}\cdot\derZ_k(e_3)
			\\&=&
			0.
		\end{eqnarray*}
		\item[$\bullet$]By using the above fact $\cnabla (e_1,e_2)e_3=0$, and Equation \eqref{Eqn: compatible condition 3 of Courant algebroid in n system}, we have
		\begin{eqnarray*}
			&&\nabla_{(\Delta_{e_3}e_2)}e_1
			=\lang{\derZ_i(e_2)}{e_3}\cdot\lang{\dere^i}{\dere^j}\cdot\derZ_j(e_1)
			=0,\\
			&&\Lang{Q(e_1,e_2,e_3)}{t}
			=\Lang{\cnabla (e_1,t)e_2}{e_3}
			+\Lang{\cnabla (e_2,t)e_3}{e_1}
			+\Lang{\cnabla (e_3,t)e_1}{e_2}
			=0.
		\end{eqnarray*}
	\item[$\bullet$] It remains show that $(\D\lang{e_1}{e_3})\circnabla  e_2=0$.  By applying Equations \eqref{Eqn: compatible condition 1 of Courant algebroid in n system}, \eqref{Eqn: compatible condition 2 of pre-Courant algebroid in n system} and \eqref{Eqn: compatible condition 2 of Courant algebroid in n system}, we get
	\begin{eqnarray*}
		(\D\lang{e_1}{e_3})\circnabla  e_2
		&=&\derz_j\big(\lang{e_1}{e_3}\big)\cdot\lang{\dere^j}{\dere^i}\cdot\derZ_i(e_2)
		\\&&\qquad
		+\Lang{e_2}{\dere^j}\cdot\BrCDO{\derz_i}{\derz_j}\big(\lang{e_1}{e_3}\big)\cdot\dere^i
		\\&&\qquad\qquad
		-\Lang{e_2}{\dere^j}\cdot\derz_k\big(\lang{e_1}{e_3}\big)\cdot\derZ_j(\dere^k)
		+\derz_k\big(\lang{e_1}{e_3}\big)\cdot\Lang{\derZ_i(\dere^k)}{e_2}\cdot \dere^i
		\\&=&\derz_j\big(\lang{e_1}{e_3}\big)\cdot\lang{\dere^j}{\dere^i}\cdot\derZ_i(e_2)
		\\&&\qquad
		+\Lang{e_2}{\dere^j}\cdot(C_{ji}^k-C_{ij}^k)\derz_k\big(\lang{e_1}{e_3}\big)\cdot\dere^i
		\\&&\qquad\qquad
		-\Lang{e_2}{\dere^j}\cdot\derz_k\big(\lang{e_1}{e_3}\big)\cdot C_{ji}^k\dere^i
		+\derz_k\big(\lang{e_1}{e_3}\big)\cdot\Lang{C_{ij}^k\dere^j}{e_2}\cdot \dere^i
		\\&=&0.
	\end{eqnarray*}
	\end{enumerate}
So, the following Leibniz identity holds:
\begin{eqnarray*}
&&e_1\circnabla (e_2\circnabla  e_3)
-(e_1\circnabla  e_2)\circnabla  e_3
-e_2\circnabla (e_1\circnabla  e_3)\\
&=&
\cnabla (e_1,e_2)e_3
+\cnabla (e_2,e_3)e_1
+\cnabla (e_3,e_1)e_2\\
\ &&
+\nabla_{(\Delta_{e_3}e_2)}e_1
+\nabla_{(\Delta_{e_1}e_3)}e_2
+\nabla_{(\Delta_{e_2}e_1)}e_3
+Q(e_1,e_2,e_3)
+(\D\lang{e_1}{e_3})\circnabla  e_2
\\&=&
0.
\end{eqnarray*}
This fact confirms that the metric algebroid $\big(E$, $\lang{\tobefilledin}{\tobefilledin}$, $\circnabla$, $a_E\big)$ is a Courant algebroid.
\end{proof}
}
\begin{Def}
A metric $n$-system
$(\derZ_1,\ldots,\derZ_n;\dere^1,\ldots,\dere^n)$ on $E$
is called \textbf{compatible} if it satisfies Conditions
\eqref{Eqn: compatible condition 1 of Courant algebroid in n system}, \eqref{Eqn: compatible condition 3 of Courant algebroid in n system} and \eqref{Eqn: compatible condition 2 of Courant algebroid in n system}.
\end{Def}
 Let us now examine some specific examples.
\begin{Ex}
Let $c_1\in\R$ be a constant. Consider   $M=\R\setminus\{c_1\}$ and the vector space $V(\cong \R^4)=\Span_{\R}\{w_1,w_2,w_3,w_4\}$. We define a vector bundle $E=M\times V$ over $M$. Assume that $V$ is equipped with a metric which is determined by the following matrix
\begin{eqnarray*}
\begin{pmatrix}
0 & 0& 1& 0 \\
0 & 0& 0& -1 \\
1 & 0& 0& 0 \\
0 & -1& 0& 0
\end{pmatrix}
\end{eqnarray*}
with  respect to the basis  $\{w_1,w_2,w_3,w_4\}$. 
Let $\derZ_1$ and $\derZ_2\in\Gamma(\CDO(E))$ be  covariant differential operators equipped with symbols   $\derz_1=\frac{d}{dt}$ and $\derz_2=0$, respectively,  such that 
\begin{eqnarray*}
&&\derZ_1 (w_i)=0,\\
&&\derZ_2\begin{pmatrix}w_1 \\ w_2 \\ w_3 \\ w_4\end{pmatrix}
=
\begin{pmatrix}
0
&\frac{c_2}{c_1-t}
&0
&0
\\
\frac{1}{c_1-t}
&\frac{c_3}{c_1-t}
&0
&0
\\
0
&\frac{c_4}{c_1-t}
&0
&\frac{1}{c_1-t}
\\
\frac{c_4}{c_1-t}
&0
&\frac{c_2}{c_1-t}
&-\frac{c_3}{c_1-t}
\end{pmatrix}
\begin{pmatrix}w_1 \\ w_2 \\ w_3 \\ w_4\end{pmatrix},
\quad\text{where}~ c_2,c_3,c_4\in\R.
\end{eqnarray*}
The collection $(\derZ_1,\derZ_2;w_1,w_2)$ forms a metric $2$-system on $E$. 
  Observing that $\BrCDO{\derZ_1}{\derZ_2}=\frac{1}{c_1-t}\derZ_2$, we can directly verify the compatibility conditions. Consequently, this metric $2$-system gives rise to a corresponding Courant algebroid structure on \( E \), which is characterized by its anchor and Dorfman bracket:
\begin{eqnarray*}
&&\quad a_E(w_3)=\frac{d}{dt},\qquad\qquad\qquad\quad
a_E(w_1)=a_E(w_2)=a_E(w_4)=0,\\
&&w_2\circnabla w_3=\frac{1}{c_1-t}w_2,\qquad\qquad\hspace*{0.2em}
w_2\circnabla w_4=\frac{1}{c_1-t}w_1,\qquad\qquad
w_3\circnabla w_4=\frac{1}{c_1-t}w_4,\\
&&w_1\circnabla w_1=
w_1\circnabla w_2=
w_1\circnabla w_3=
w_1\circnabla w_4=
w_2\circnabla w_2=
w_3\circnabla w_3=
w_4\circnabla w_4=0.
\end{eqnarray*}

\end{Ex}

		\begin{Ex} 
		Let $(\derZ_1,\ldots,\derZ_n;\dere^1,\ldots,\dere^n)$ be a compatible metric $n$-system on $(E,\lang{\tobefilledin}{\tobefilledin})$  and $A=M\times\mathbb{R}^n$ be a trivial vector bundle over $M$ with standard frame $\{\varepsilon_1,\ldots,\varepsilon_n\}$ and dual frame $\{\varepsilon^1,\ldots,\varepsilon^n\}$.  Consider  the vector bundle maps $$
		\iota:A\longrightarrow E,
		\quad \iota(\varepsilon_i)=\dere^i\quad
		\mbox{and}\quad
		\sigma:A^*\longrightarrow TM,
		\quad \sigma(\varepsilon^i)=\widehat z_i.$$  
		Using the pseudo-metric to identify $E$ with $E^*$, define
		$\iota^\dagger:E\to A^*$ as the composition of this identification
		with the dual map $\iota^*:E^*\to A^*$, namely, $\iota^\dagger(e)
		=\sum_{i=1}^{n}\lang{e}{\dere^i}\varepsilon^i$, for all $e\in \Gamma(E)$.
		Assume that
		$L:=\iota(A)
		=\operatorname{span}_{C^\infty(M)}
		\{\widehat u^1,\ldots,\widehat u^n\}.$
		The condition \eqref{Eqn: compatible condition 3 of Courant algebroid in n system}
		is equivalent to \(\iota^\dagger\circ\iota=0\), and {hence \(L\) is an
			isotropic subbundle of \(E\).} 
		Using conditions \eqref{Eqn: compatible condition 1 of Courant algebroid in n system} and \eqref{Eqn: compatible condition 3 of Courant algebroid in n system}, we obtain
		\begin{align*}
			\lambda\circnabla\eta
			=\lang{\lambda}{\widehat u^i}\derZ_i(\eta)-\lang{\eta}{\widehat u^i}\widehat Z_i(\lambda)+\lang{\widehat Z_i(\lambda)}{\eta}\widehat u^i
			=\lang{\widehat z_i(f^j)\widehat u^j
				+f^jC_{ij}^{l}\widehat u^l}{g^k\widehat u^k}\widehat u^i=0,
		\end{align*}
		where \(\lambda=f^j\widehat u^j\) and \(\eta=g^k\widehat u^k\).
		Therefore $\Gamma(L)\circnabla\Gamma(L)=0.$
		Furthermore assume that \(\widehat z_1,\ldots,\widehat z_n\) are pointwise linearly independent,
		then \(\sigma:A^*\to F\) is an isomorphism onto $F:=\operatorname{span}\{\widehat z_1,\ldots,\widehat z_n\}$
		and \(\sigma^*:T^*M\to A\) is surjective. 
		By \eqref{Eqn: compatible condition of pseudo-metric and CDO}, we can define a metric $F$-connection on \(E\) by
		\begin{equation}
			\nabla^F_{\widehat z_i}e:=\widehat Z_i(e),
			\qquad
			\nabla^F_{f^i\widehat z_i}e:=f^i\widehat Z_i(e).
			\label{eq:partial-connection}
		\end{equation}
		The conditions
		\eqref{Eqn: compatible condition 1 of Courant algebroid in n system} and \eqref{Eqn: compatible condition 2 of pre-Courant algebroid in n system} imply that this connection $\nabla^F$ preserves $A$ and  the distribution \(F\) is involutive.
		The stronger condition \eqref{Eqn: compatible condition 2 of Courant algebroid in n system} gives rise to a flat connection, i.e.,
		\begin{align*}
			R^{\nabla^F}(\widehat z_i,\widehat z_j)e
			={}
			\nabla^F_{\widehat z_i}\nabla^F_{\widehat z_j}e
			-\nabla^F_{\widehat z_j}\nabla^F_{\widehat z_i}e
			-\nabla^F_{[\widehat z_i,\widehat z_j]}e
			=\big(
			[\widehat Z_i,\widehat Z_j]
			c_{ji}^{k}-c_{ij}^{k}\widehat Z_k
			\big)(e)
			=0.
		\end{align*}
		In summary, in this case a compatible metric \(n\)-system consists
		geometrically of a framed isotropic subbundle $L$, an involutive
	 distribution $F$, and a flat metric  connection $\nabla^F$ which preserves the isotropic subbundle. The construction is
		restrictive globally, since it requires a global frame and a flat lift, but it
		contains the standard local geometry of exact Courant algebroids, as the next
		example shows.
	\end{Ex}
	\begin{Ex}
		Let $M$ be an $n$-dimensional smooth manifold with  local coordinates $\{x_1,\cdots,x_n\}$. Consider the vector bundle $E=TM\oplus T^*M$ with the standard pairing \eqref{Eqn:standard pairing}. Assume that $\derZ_i$, $i=1,2,\cdots,n$ are covariant differential operators on $E$ equipped with symbols  $\derz_i:=\frac{\partial}{\partial x_i}$ such that
		\begin{eqnarray*}
			\derZ_i(f_j\frac{\partial}{\partial x_j}+g^jd x_j)=\frac{\partial f_j}{\partial x_i}\frac{\partial}{\partial x_j}+\frac{\partial g^j}{\partial x_i}d x_j,\quad \forall f_j,g^j\in C^\infty(M).
		\end{eqnarray*}
		The collection $(\derZ_1,\derZ_2,\cdots,\derZ_n;dx_1,dx_2,\cdots,dx_n)$ forms a metric $n$-system on $E$.  Therefore, $E=TM\oplus T^*M$ forms a Courant algebroid. The anchor map  of $E$ is explicitly given by
		\begin{eqnarray*}
			a_E(\frac{\partial}{\partial x_i})=\frac{\partial}{\partial x_i},\quad a_E(dx_i)=0.
		\end{eqnarray*}
		The Dorfman bracket of $E$ is generated by the trivial relations
		\begin{eqnarray*}
			\frac{\partial}{\partial x_i} \circnabla \frac{\partial}{\partial x_j}=0,\quad \frac{\partial}{\partial x_i}\circnabla dx_j=0,\quad dx_i\circnabla dx_j=0.
		\end{eqnarray*}
		However, the Dorfman bracket on $\Gamma(E)$ is not trivial. In fact, for any $X+\alpha=f^1_i\frac{\partial}{\partial x_i}+g_1^idx_i$, $Y+\beta=f^2_i\frac{\partial}{\partial x_i}+g_2^idx_i\in E$, we have
		\begin{eqnarray*}
			X\circnabla Y&=&(f^1_i\frac{\partial f^2_j}{\partial x_i}-f^2_i\frac{\partial f^1_j}{\partial x_i})\frac{\partial}{\partial x_j},\quad X\circnabla \beta=(f^1_i\frac{\partial g_2^j}{\partial x_i}+g_2^i\frac{\partial f^1_i}{\partial x_j})dx_j,\\
			\alpha \circnabla Y&=&(-f^2_i\frac{\partial g_1^j}{\partial x_i}+f^2_i\frac{\partial g_1^i}{\partial x_j})dx_j,\quad\alpha \circnabla\beta =0.
		\end{eqnarray*}
		Note that, the Dorfman bracket $\circnabla$ is the standard Courant bracket \eqref{Def: Dorfman bracket of MMstar}.
	\end{Ex}
We now give a global example that does not depend on the choice of coordinates.
	\begin{Ex}
	Let $M=\mathbb{T}^n=\mathbb{R}^n/\mathbb{Z}^n$ and
	$q:\mathbb{R}^n\to M$ be the quotient map. Choose a basis
	$\{v_1,\ldots,v_n\}$ of $\mathbb{R}^n$ and define vector
	fields $V_i$ on $M$ by
	\[
	V_i(q(x))=dq_x(v_i),
	\]
	where $T_x\mathbb{R}^n$ is canonically identified with
	$\mathbb{R}^n$. For each $k\in\mathbb{Z}^n$, the translation
	$\tau_k(x)=x+k$ satisfies $q\circ\tau_k=q$, and hence
	\[
	dq_{x+k}(v_i)
	=dq_{x+k}\bigl(d(\tau_k)_x(v_i)\bigr)
	=dq_x(v_i).
	\]
	Thus the $V_i$ are well defined and form a global frame of
	$TM$.
	Let $\{\theta^1,\ldots,\theta^n\}$ be the dual frame of
	$T^*M$. Then
	\[
	[V_i,V_j]=0,\qquad
	\theta^i(V_j)=\delta^i_j,\qquad
	d\theta^i=0.
	\]
	
	Consider the vector bundle $E=TM\oplus T^*M$ with the
	standard pairing. Define covariant differential operators
	$\widehat Z_i$, $i=1,\ldots,n$, with symbols
	$\widehat z_i=V_i$ by
	\[
	\widehat Z_i(f^jV_j+g_j\theta^j)
	=V_i(f^j)V_j+V_i(g_j)\theta^j,
	\qquad f^j,g_j\in C^\infty(M).
	\]
	For all $e_1,e_2\in\Gamma(E)$, these operators satisfy
	\[
	\lang{\derZ_i(e_1)}{e_2}
	+\lang{e_1}{\derZ_i(e_2)} 
	=V_i\lang{e_1}{e_2}.
	\]
	The collection
	$(\derZ_1,\ldots,\derZ_n;
	\theta^1,\ldots,\theta^n)$ therefore forms a metric
	$n$-system on $E$. Moreover,
	\[
	\derZ_i(\theta^j)=0,\qquad
	\lang{\theta^i}{\theta^j}=0,\qquad
	[\derZ_i,\derZ_j]=0.
	\]
	Hence  $E$ becomes a Courant algebroid.
	Its anchor is explicitly given by
	\[
	a_E(V_i)=V_i,\qquad a_E(\theta^i)=0.
	\]
	The associated Dorfman bracket satisfies
	\[
	V_i\mathbin{\overset{\nabla}{\circ}}V_j=0,\qquad
	V_i\mathbin{\overset{\nabla}{\circ}}\theta^j=0,\qquad
	\theta^i\mathbin{\overset{\nabla}{\circ}}V_j=0,\qquad
	\theta^i\mathbin{\overset{\nabla}{\circ}}\theta^j=0.
	\]
	However, this bracket is not trivial. Indeed, for any
	$X+\alpha=f^iV_i+g_i\theta^i$ and
	$Y+\beta=h^iV_i+k_i\theta^i$ in $\Gamma(E)$, we have
	\[
	\begin{aligned}
		X\mathbin{\overset{\nabla}{\circ}}Y
		&=\bigl(f^iV_i(h^j)-h^iV_i(f^j)\bigr)V_j,\quad
		X\mathbin{\overset{\nabla}{\circ}}\beta\quad =\bigl(f^iV_i(k_j)+k_iV_j(f^i)\bigr)\theta^j,\\
		\alpha\mathbin{\overset{\nabla}{\circ}}Y
		&=\bigl(-h^iV_i(g_j)+h^iV_j(g_i)\bigr)\theta^j,\quad
		\alpha\mathbin{\overset{\nabla}{\circ}}\beta=0.
	\end{aligned}
	\]
	In fact, this is the standard Dorfman bracket:
	\[
	(X+\alpha)\mathbin{\overset{\nabla}{\circ}}(Y+\beta)
	=[X,Y]+\mathcal L_X\beta-\iota_Yd\alpha.
	\]
	
	Now choose another basis $v'_i=A_i{}^jv_j$ of
	$\mathbb{R}^n$, where $A\in\mathrm{GL}(n,\mathbb{R})$
	is constant. The corresponding frames and operators satisfy
	\[
	V'_i=A_i{}^jV_j,\qquad
	\theta'^i=(A^{-1})_j{}^i\theta^j,\qquad
	\widehat Z'_i=A_i{}^j\widehat Z_j.
	\]
	Thus
	$(\widehat Z'_1,\ldots,\widehat Z'_n;
	\theta'^1,\ldots,\theta'^n)$ is again a compatible metric
	$n$-system. Furthermore,
	\[
	\begin{aligned}
		\sum_i\langle e\mid\theta'^i\rangle V'_i
		&=\sum_i\langle e\mid\theta^i\rangle V_i,\\
		\sum_i\langle e_1\mid\theta'^i\rangle\widehat Z'_i(e_2)
		&=\sum_i\langle e_1\mid\theta^i\rangle\widehat Z_i(e_2),\\
		\sum_i\langle\widehat Z'_i(e_1)\mid e_2\rangle\theta'^i
		&=\sum_i\langle\widehat Z_i(e_1)\mid e_2\rangle\theta^i.
	\end{aligned}
	\]
	Consequently, the induced anchor, metric connection and
	Dorfman bracket are unchanged and then the 
	Courant algebroid structure is independent of the chosen
	basis of $\mathbb{R}^n$ and is expressed intrinsically
	by the standard Dorfman bracket above.
	\end{Ex}

\subsection{Construction of  Lie bialgebroids from  compatible metric $n$-systems}\label{Lie bialgebroids from metric}
Suppose that $(\derX_1,\ldots,\derX_p;$ $\derad^1,\ldots,\derad^p)$ is a compatible $p$-system on the vector bundle $A$ such that $A$ is a Lie algebroid. Suppose that $(\derY_1,\ldots,\derY_q;\dera^1,\ldots,\dera^q)$ is a compatible $q$-system on the vector bundle $A^*$ such that $A^*$ is also a Lie algebroid. Our objective is to obtain the sufficient conditions for the two Lie algebroids, $A$ and $A^*$, to form a Lie bialgebroid. The operators $\derX_1, \derX_2, \ldots, \derX_p$ and $\derY_1, \derY_2, \ldots, \derY_q$ can be expressed as covariant differential operators acting on both $A$ and $A^*$. Additionally, they naturally extend to covariant differential operators on $\AAstar$. This extension preserves the standard pairing as defined in Equation \eqref{Eqn:standard pairing}.

		Suppose that
	\begin{eqnarray*}
		\derX_i(\derad^j)
		=
		a_{ik}^j \derad^k
		,
		\qquad
		\derY_\alpha(\dera^\beta)
		=
		b_{\alpha \gamma}^\beta \dera^\gamma
		,
		\qquad
		\derX_i(\dera^\alpha)
		=
		c_{i\gamma}^\alpha\dera^\gamma
		,
		\qquad
		\derY_\alpha(\derad^i)
		=
		d_{\alpha k}^i\derad^k
		,
	\end{eqnarray*}
	for some $a_{ik}^j, b_{\alpha\gamma}^\beta, c_{i\gamma}^\alpha, d_{\alpha k}^i \in\CinfM$. By the compatibility conditions, we already have
	\begin{eqnarray}
	\BrCDO{\derX_i}{\derX_j}
&\hspace*{-0.5em}=\hspace*{-0.5em}&
	(a_{ji}^k-a_{ij}^k)\derX_k
	,\label{Eqn: compatible condition 1 of Courant algebroid AAstar in n system}\\
	\BrCDO{\derY_\alpha}{\derY_\beta}
&\hspace*{-0.5em}=\hspace*{-0.5em}&
	(b_{\beta \alpha}^\gamma-b_{\alpha \beta}^\gamma)\derY_\gamma
	.\label{Eqn: compatible condition 2 of Courant algebroid AAstar in n system}\end{eqnarray}
	We further impose two conditions
	\begin{eqnarray}
	\label{Eqn: compatible condition 3 of Courant algebroid AAstar in n system}
	\BrCDO{\derY_\alpha}{\derX_i}
&\hspace*{-0.5em}=\hspace*{-0.5em}&
	c_{i\alpha}^\gamma \derY_\gamma-d_{\alpha i}^k \derX_k
	,
	\\
	\label{Eqn: compatible condition 4 of Courant algebroid AAstar in n system}
	\text{and}\quad \Lang{\derad^i}{\dera^\alpha}
&\hspace*{-0.5em}=\hspace*{-0.5em}&0
	.
	\end{eqnarray}
	It is easy to see that under the extra conditions \eqref{Eqn: compatible condition 3 of Courant algebroid AAstar in n system} and \eqref{Eqn: compatible condition 4 of Courant algebroid AAstar in n system}, the set of data $$(\derX_1, \ldots, \derX_p, \derY_1, \ldots, \derY_q; \derad^1, \ldots, \derad^p, \dera^1, \ldots, \dera^q)$$ defines a compatible metric $(p+q)$-system. Consequently, $E=A\oplus A^*$ becomes a Courant algebroid. In other words,  $A$ and $A^*$ form a Lie bialgebroid.
\begin{Ex}
Let $M=\R $ be a 1-dimensional smooth manifold. Consider the 2-dimensional vector space $V=\R^2=\Span_{\R}\{w_1,w_2\}$ and its dual space $V^*=\Span_{\R}\{w^1,w^2\}$. We define Lie algebroids $A=M\times V$
and $A^*=M\times V^*$ over $M$.
Assume that
$\derX:\secAstar\rightarrow\secAstar$ is a covariant differential operator with symbol $\derx=\frac{d}{dt}$ such that
\begin{eqnarray*}
\derX\begin{pmatrix}w^1 \\ w^2 \end{pmatrix}
=\begin{pmatrix}-k &0 \\ 0 &0 \end{pmatrix}
\begin{pmatrix}w^1 \\ w^2 \end{pmatrix}
,\qquad\text{for some}\quad k\in\R.
\end{eqnarray*}
Assume that $\derY:\secA\rightarrow\secA$ is a covariant differential operator with symbol $\dery=\frac{d}{dt}$ such that
\begin{eqnarray*}
\derY\begin{pmatrix}w_1 \\ w_2\end{pmatrix}
=
\begin{pmatrix}
0
&0
\\
0
&-l
\end{pmatrix}
\begin{pmatrix}w_1 \\ w_2 \end{pmatrix},
\qquad\text{for some}\quad l\in\R.
\end{eqnarray*}
Then $(\derX;w^1)$ and $(\derY;w_2)$ are compatible metric $1$-systems on $A$ and $A^*$, respectively.
Moreover, we have
\begin{eqnarray*}
\derX(w^1)=-kw^1,\quad
\derY(w_2)=-lw_2,\quad
\derX(w_2)=0,\quad
\derY(w^1)=0,\quad
\BrCDO{\derY}{\derX}=0,\quad
\lang{w^1}{w_2}=0.
\end{eqnarray*}
It is easy to check that $(\derX,\derY;w^1,w_2)$ satisfies Conditions \eqref{Eqn: compatible condition 1 of Courant algebroid AAstar in n system}---\eqref{Eqn: compatible condition 4 of Courant algebroid AAstar in n system}.     Consequently,  $E=A\oplus A^*$ forms a Courant algebroid. The  anchor map of $E$ is explicitly given by
\begin{eqnarray*}
a_E(w_1)=\frac{d}{dt}\quad\text{and}\quad
a_E(w_2)=0.
\end{eqnarray*}
The Dorfman bracket of $E$  is generated by the  trivial relations 
\begin{eqnarray*}
	&&w^1\circnabla w^1=w^1\circnabla w^2=w^1\circnabla w_1=w^1\circnabla w_2=0,\qquad w^2\circnabla w^1=w^2\circnabla w^2=w^2\circnabla w_1=w^2\circnabla w_2=0,\\
	&&w_1\circnabla w^1=w_1\circnabla w^2=w_1\circnabla w_1=w_1\circnabla w_2=0,\qquad w_2\circnabla w^1=w_2\circnabla w^2=w_2\circnabla w_1=w_2\circnabla w_2=0.
\end{eqnarray*}
However, we cannot say that the Dorfman bracket on $E$ is trivial. For example, we have
  \[
  w_1\circnabla fw_1=\frac{df}{dt}w_1,\quad
  w_1\circnabla fw_2=\frac{df}{dt}w_2,\quad \forall f\in C^\infty(M).
  \]
\end{Ex}

\subsection{Compatible metric $n$-systems and Manin pairs}\label{Manin pairs and dorfman connections}
In this subsection, we construct Manin pairs from certain compatible metric $n$-systems, find a formula to produce Dorfman connection, and study the associated Lie algebroids.
\subsubsection{Construction of  Manin pairs}\label{construction of manin pairs}

\begin{Def}\label{Dirac structure}
	\cite{BP}
	Let $\big(E,\lang{\tobefilledin}{\tobefilledin},\circ,a_E\big)$ be a Courant algebroid over $M$. A \textbf{Dirac structure} is a subbundle $L\subset E$ that is maximally isotropic with respect to the pseudo-metric $\lang{\tobefilledin}{\tobefilledin}$ and the space $\Gamma(L)$ of smooth sections of $L$ is closed under the bracket $\circ$, i.e.,
	$\Gamma(L)\circ\Gamma(L) \subseteq \Gamma(L)$. In this case, $L$ naturally constitutes a Lie algebroid.  The pair $(E,L)$ is called a \textbf{Manin pair} over $M$.
\end{Def}
Consider a compatible metric $n$-system $(\derZ_1$, $\ldots$, $\derZ_n$; $\dere^1$, $\ldots$, $\dere^n)$  on a metric bundle $\big(E,\lang{\tobefilledin}{\tobefilledin})$, and thus $E$ is endowed with a Courant algebroid structure according to Proposition \ref{Courant algebroid}.

We suppose that the following three conditions hold:
\begin{enumerate}
	\item[(i)]The covariant differential operators   $(\derZ_1,\ldots,\derZ_n)$   can be  divided into two parts: $(\derZ_1,\ldots,\derZ_k)$ and $(\derZ_{k+1},\ldots,\derZ_n)$. Assume that the second part $(\derZ_{k+1}, \ldots, \derZ_n) $  is closed under the commutator operation:  
	\begin{eqnarray}
		\label{Eqn: compatible condition of Manin pair}
		\BrCDO{\derZ_{a}}{\derZ_{b}}
		=\sum_{c=k+1}^n
		(C_{ba}^{c}-C_{ab}^{c})\derZ_{c},
		\qquad\forall a,b= k+1, k+2,\ldots,n,
	\end{eqnarray}
	where the structure functions $C_{ab}^c \in \CinfM$ are subject to Equation \eqref{Eqn: compatible condition 1 of Courant algebroid in n system}. 
	This Equation \eqref{Eqn: compatible condition of Manin pair} can be equivalently described as    the following condition:
	\begin{eqnarray}\label{Eqn: compatible condition of Manin pair equivalent}
		C_{ba}^c
		=C_{ab}^c
		,
		\qquad
		\forall ~a,b= k+1, k+2,\ldots,n,~
		c=1,2,\ldots,k.
	\end{eqnarray}
	 
	\item[(ii)] Assume that $L$ is given by  
	\begin{eqnarray*}
		\{v\in E|\lang{v}{\dere^1}=\lang{v}{\dere^2}=\cdots=\lang{v}{\dere^k}=0\}.
	\end{eqnarray*}.
	\item[(iii)] Assume that $L$ constitutes a maximal isotropic subbundle of $E$.
\end{enumerate}
 By the definition of $L$ and Condition \eqref{Eqn: compatible condition 3 of Courant algebroid in n system}, every $\dere^j$ belongs to $\Gamma(L)$. Since the set $\{\dere^1,\cdots,\dere^n\}$ is linearly independent over $C^\infty(M)$ and $L$ is a maximal isotropic subbundle of $E$, the  rank of   $E$ is at least $2n$.
\begin{Cor}
	Under the above three conditions (i), (ii), and (iii),	$(E,L)$ is a Manin pair. Moreover, the Lie algebroid  structure maps  of $L$ are given by
	\begin{eqnarray*}
		a_L(v)
		&:=\hspace*{-0.5em}&a_E(v)
		=\sum_{l=k+1}^n \lang{v}{\dere^l}\derz_l
		,
		\qquad\forall v\in L,
		\\
		\BrL{v_1}{v_2}
		&:=\hspace*{-0.5em}&v_1\circn v_2
		\\&\equalbyreason{\eqref{Def: Dorfman bracket of Courant algebroid in n system}}&\sum_{l=k+1}^n\big(
		\Lang{v_1}{\dere^l}\cdot\derZ_l(v_2)
		-\Lang{v_2}{\dere^l}\cdot\derZ_l(v_1)\big)
		+\sum_{l=1}^n\Lang{\derZ_l(v_1)}{v_2}\cdot \dere^l
		,
		\qquad
		\forall v_1,v_2\in \Gamma(L).
	\end{eqnarray*}
\end{Cor}
	\begin{proof}
		It suffices  to verify that $\BrL{v_1}{v_2}\in \Gamma(L)$. In fact, for any 
		$j=1,\cdots,k$, we use the above expression of $\BrL{v_1}{v_2}$  to get 
		\begin{eqnarray*}
			\lang{\BrL{v_1}{v_2}}{\dere^j}&\equalbyreason{\eqref{Eqn: compatible condition 3 of Courant algebroid in n system}}&\sum_{l=k+1}^n
			\lang{v_1}{\dere^l}\lang{\derZ_l(v_2)}{\dere^j}-\sum_{l=k+1}^n\,\lang{v_2}{\dere^l}\lang{\derZ_l(v_1)}{\dere^j}\\
			&\equalbyreason{\eqref{Eqn: compatible condition of pseudo-metric and CDO}}&\sum_{l=k+1}^n\lang{v_1}{\dere^l}(\derz_l\lang{v_2}{\dere^j}-\lang{v_2}{\derZ_l(\dere_j)})-\sum_{l=k+1}^n\lang{v_2}{\dere^l}(\derz_l\lang{v_1}{\dere^j}-\lang{v_1}{\derZ_l(\dere^j)})\\
			&=&\sum_{l=k+1}^n(-\lang{v_1}{\dere^l}\lang{v_2}{\derZ_l(\dere^j)}+\lang{v_2}{\dere^l}\lang{v_1}{\derZ_l(\dere^j)})\\
			&\equalbyreason{\eqref{Eqn: compatible condition 1 of Courant algebroid in n system}}&\sum_{l_1,l_2=k+1}^n(-\lang{v_1}{\dere^{l_1}}C_{l_1l_2}^j\lang{v_2}{\dere^{l_2}}+\lang{v_2}{\dere^{l_2}}C_{l_1l_2}^j\lang{v_1}{\dere^{l_1}})\\
			&=&\sum_{l_1,l_2=k+1}^n(C_{l_2l_1}^j-C_{l_1l_2}^j)\lang{v_1}{\dere^{l_1}}\lang{v_2}{\dere^{l_2}}\equalbyreason{\eqref{Eqn: compatible condition of Manin pair equivalent}}0.
		\end{eqnarray*}
		This completes the proof.
	\end{proof}
	 
\subsubsection{The associated Dorfman connection}
 Let us now study the relationship between Dorfman connections and compatible metric $n$-systems.  The following Definitions \ref{Def:temp1}, \ref{Def:temp2} and Proposition \ref{Prop:temp1}  are recalled from \cite{MJ2018}.

\begin{Def}\label{Def:temp1}
	Let $\big(Q,[\tobefilledin,\tobefilledin]_Q,a_Q\big)$ be a dull algebroid over $M$. Let $B\rightarrow M$ be a vector bundle equipped with a fiberwise pairing
	$\langle \tobefilledin,\tobefilledin\rangle_{QB}:
	Q{\times_M} B\rightarrow\mathbb{R}$ and a map
	$\dd_B: C^\infty(M)\rightarrow\Gamma(B)$ such that
	\begin{eqnarray*}
		\langle q,\dd_Bf\rangle_{QB}=a_Q(q)(f),
	\end{eqnarray*}
	for all $q\in\Gamma(Q)$ and $f\in C^\infty(M)$.
	Then $\big(B,\dd_B,\langle \tobefilledin,\tobefilledin\rangle_{QB}\big)$
	is called a pre-dual of $Q$. 
\end{Def}

\begin{Def} \label{Def:temp2}
	Let $\big(Q,[\tobefilledin,\tobefilledin]_Q,a_Q\big)$ be a dull algebroid over $M$ and $\big(B,\dd_B,\langle \tobefilledin,\tobefilledin\rangle_{QB}\big)$ a pre-dual of $Q$.
	A Dorfman $(Q$-$)$connection on $B$ is an $\mathbb{R}$-bilinear map
	\begin{eqnarray*}
		\nabla^{\mathrm{Dorf}}:\Gamma(Q)\times\Gamma(B)\rightarrow\Gamma(B)
	\end{eqnarray*}
	such that
	\begin{itemize}
		\item[$(a)$]
		$\nabla^{\mathrm{Dorf}}_{fq}b=f\nabla^{\mathrm{Dorf}}_qb+\langle q,b\rangle_{QB}\cdot \dd_Bf$,
		\item[$(b)$]
		$\nabla^{\mathrm{Dorf}}_q(fb)=f\nabla^{\mathrm{Dorf}}_qb+a_Q(q)(f)\cdot b$,
		\item[$(c)$]
		$\nabla^{\mathrm{Dorf}}_q(\dd_Bf)=\dd_B\Big(a_Q(q)(f)\Big)$,
	\end{itemize}
	for all $q\in\Gamma(Q), b\in\Gamma(B), f\in C^\infty(M)$.
\end{Def}

Let $L \subseteq E$ be an  isotropic subalgebroid. This means that $L$ is isotropic with respect to $\lang{\tobefilledin}{\tobefilledin}$ (possibly not maximal) and $\Gamma(L)$ is closed with respect to $\circ$. In fact, a (possibly non-unique) Dorfman connection exists on the quotient bundle $E/L$.
 \begin{prop} \label{Prop:temp1}With assumptions as above, $L$ has a Lie algebroid structure induced from $E$ and the map
	\begin{eqnarray*}
		\nabla^{\mathrm{Dorf}}: \Gamma(L)\times\Gamma(E/L)\rightarrow\Gamma(E/L), \qquad
		\nabla^{\mathrm{Dorf}}_v\bar{e}=\overline{v\circ e},
		\qquad\forall v\in\Gamma(L), \bar{e}\in\Gamma(E/L),
	\end{eqnarray*}
	is an  $L$-Dorfman connection. Here the map $\dd_{E/L}: C^\infty(M)\rightarrow\Gamma(E/L)$ is induced from $\D: C^\infty(M)\rightarrow\Gamma(E)$ and the pairing
	$\langle \tobefilledin,\tobefilledin\rangle_{L, E/L}:
	L{\times_M}(E/L)\rightarrow\mathbb{R}$ is the natural pairing induced by the pairing on E, i.e.,
	\begin{eqnarray*}
		~[v_1,v_2]_L=[v_1,v_2]_E,\quad
		a_L(v)=a_E(v),\quad
		\langle v,\bar{e}\rangle_{L,E/L}=\lang{v}{e},\quad
		\dd_{E/L}(f)=\overline{\D(f)} ,
	\end{eqnarray*}
	for all $v,v_1,v_2\in\Gamma(L), e\in\Gamma(E)$.

In particular, arising from any Manin pair $(E,L)$ there is an $L$-Dorfman connection on $E/L$.
\end{prop} 
We now consider a compatible metric $n$-system $(\derZ_1$, $\ldots$, $\derZ_n$; $\dere^1$, $\ldots$, $\dere^n)$  on a metric bundle $(E,\lang{\tobefilledin}{\tobefilledin})$, and suppose that Conditions (i), (ii), and (iii) in Section \ref{construction of manin pairs} are satisfied. Thus we have a Manin pair $(E,L)$ accordingly.  Using Equation \eqref{Def: Dorfman bracket of Courant algebroid in n system}, the associated Dorfman connection can be explicitly presented.
\begin{Cor}\label{manin pair wrt dorfman connection}
The  Dorfman $(L$-$)$connection
	$
	\nabla^{Dorf}:\Gamma(L)\times\Gamma(E/L)\rightarrow\Gamma(E/L)
	$
	is given by
	\begin{eqnarray}\label{Dorfman connection from manin pair}
		\nabla^{\mathrm{Dorf}}_v\bar{e}
		=\overline{v\circ e}
		=-\sum_{i=1}^n\Lang{e}{\dere^i}\cdot\overline{\derZ_i(v)},
	\end{eqnarray}
	for all $v\in \Gamma(L)$ and $e\in \Gamma(E)$.
\end{Cor}
\begin{Ex}
Let $C\in\R$ be a constant.
Consider  $M=\R\setminus\{C\}$ and the vector space 
$V(\cong \R^8)=\Span_{\R}\{w_1,w_2,w_3,w_4,w_5,w_6,w_7,w_8\}$. We define a vector bundle $E=M\times V$ over $M$.
Assume that $V$ is equipped with a metric which is determined by the following matrix
\begin{eqnarray*}
\begin{pmatrix}
0 & 0& 0& 0& 1& 0 & 0 & 1\\
0 & 0& 0& 0& 0& 1 & 0 & 0\\
0 & 0& 0& 0& 0& 0 & 1 & 0\\
0 & 0& 0& 0& 0& 0 & 1 & 1\\
1 & 0& 0& 0& 0& 0 & 0 & 0\\
0 & 1& 0& 0& 0& 0 & 0 & 0\\
0 & 0& 1& 1& 0& 0 & 0 & 0\\
1 & 0& 0& 1& 0& 0 & 0 & 0
\end{pmatrix}
\end{eqnarray*}
with respect to the basis $\{w_1,w_2,w_3,w_4,w_5,w_6,w_7,w_8\}$.

 Suppose that $\derZ_1$ and $\derZ_2$ $\in\Gamma(\CDO(E))$ are  covariant differential operators on $E$ (both with trivial symbols) determined by the following matrices:
\begin{eqnarray*}
\derZ_1\begin{pmatrix}w_1 \\ w_2 \\ w_3 \\ w_4 \\ w_5 \\ w_6\\ w_7 \\ w_8\end{pmatrix}
&\hspace*{-0.5em}=\hspace*{-0.5em}&
\begin{pmatrix}
\frac{CC_1+C_2}{C-t}&0&\frac{1}{C-t}&0&0&0&0&0\\
\frac{C_1t+C_2}{C-t}&C_1&\frac{1}{C-t}&0&0&0&0&0\\
\frac{C_3}{C-t}&0&0&0&0&0&0&0\\
0&0&0&0&0&0&\frac{C_4}{C-t}&-\frac{C_4}{C-t}\\
0&0&0&0&-\frac{CC_1+C_2}{C-t}&-\frac{C_1t+C_2}{C-t}&-\frac{C_3}{C-t}&0\\
0&0&0&0&0&-C_1&0&0\\
0&0&0&-\frac{C_4}{C-t}&-\frac{1}{C-t}&-\frac{1}{C-t}&0&0\\
0&0&0&\frac{C_4}{C-t}&0&0&0&0
\end{pmatrix}
\begin{pmatrix}w_1 \\ w_2 \\ w_3 \\ w_4 \\ w_5 \\ w_6\\ w_7 \\ w_8\end{pmatrix},
\\
\derZ_2\begin{pmatrix}w_1 \\ w_2 \\ w_3 \\ w_4 \\w_5 \\ w_6\\ w_7 \\ w_8\end{pmatrix}
&\hspace*{-0.5em}=\hspace*{-0.5em}&
\begin{pmatrix}
0&0&0&0&0&0&0&0\\
C_1&-C_1&0&0&0&0&0&0\\
0&0&0&0&0&0&0&0\\
0&0&0&0&0&0&0&0\\
0&0&0&0&0&-C_1&0&0\\
0&0&0&0&0&C_1&0&0\\
0&0&0&0&0&0&0&0\\
0&0&0&0&0&0&0&0
\end{pmatrix}
\begin{pmatrix}w_1 \\ w_2 \\ w_3 \\ w_4 \\ w_5 \\ w_6\\ w_7 \\ w_8\end{pmatrix},
\end{eqnarray*}
where $C_1$, $C_2$, $C_3$, and $C_4$  are constants ($\in \mathbb{R}$).
 Suppose that $\derZ_3\in\Gamma(\CDO(E))$ is a covariant differential operator on $E$ equipped with the symbol   $\derz_3=\frac{d}{dt}$ such that $\derZ_3 (w_i)=0$ for all $i=1,2,\ldots,8$.

Then we can check the   relations 
$ \BrCDO{\derZ_1}{\derZ_2}=0$, 
$ \BrCDO{\derZ_1}{\derZ_3}=-\frac{1}{C-t}\derZ_1-\frac{1}{C-t}\derZ_2$,
 and 
$ \BrCDO{\derZ_2}{\derZ_3}=0$. Thus the $3$-system $(\derZ_1,\derZ_2,\derZ_3;$ $w_1,w_2,w_3)$    satisfies  Conditions
\eqref{Eqn: compatible condition 1 of Courant algebroid in n system}, \eqref{Eqn: compatible condition 3 of Courant algebroid in n system} and \eqref{Eqn: compatible condition 2 of Courant algebroid in n system}. Hence, $E$ is equipped with a Courant algebroid structure:
\begin{eqnarray*}
&&a_E(w_1)=a_E(w_2)=a_E(w_3)=a_E(w_4)=a_E(w_5)=a_E(w_6)=a_E(w_8)=0, \qquad
a_E(w_7)=\frac{d}{dt},
\\
&&
w_1\circnabla w_5=-\frac{1}{C-t}w_3,\qquad
w_1\circnabla w_7=\frac{1}{C-t} w_1,\quad
w_1\circnabla w_8=-\frac{1}{C-t}w_3,
\\
&&
w_2\circnabla w_5=-\frac{1}{C-t}w_3,\qquad
w_2\circnabla w_7=\frac{1}{C-t}w_1,\quad
w_2\circnabla w_8=-\frac{1}{C-t}w_3,\quad
w_4\circnabla w_5=\frac{C_4}{C-t}(w_8-w_7),
\\
&&
w_4\circnabla w_8=\frac{C_4}{C-t}(w_8-w_7),\hspace*{1em}
w_5\circnabla w_7=-\frac{1}{C-t}(w_5+w_6)-\frac{C_4}{C-t}w_4,
\\
&&
w_5\circnabla w_8=\frac{C_4}{C-t}w_4+\frac{CC_1+C_2}{C-t}w_5+\frac{C_1t+C_2}{C-t}w_6+\frac{C_3}{C-t}w_7,\quad
w_6\circnabla w_8=-C_1w_6,
\\
&&
w_7\circnabla w_7=-\frac{C_4}{C-t}w_1,\quad
w_7\circnabla w_8=\frac{1}{C-t}(w_5+w_6),\qquad w_8\circnabla w_8=\frac{C_4}{C-t}w_1,
\\
&&
w_1\circnabla w_1=
w_1\circnabla w_2=
w_1\circnabla w_3=
w_1\circnabla w_4=
w_1\circnabla w_6=
0\\
&&
w_2\circnabla w_2=
w_2\circnabla w_3=
w_2\circnabla w_4=
w_2\circnabla w_6=
w_3\circnabla w_3=
0
,
\\
&&
w_3\circnabla w_4=
w_3\circnabla w_5=
w_3\circnabla w_6=
w_3\circnabla w_7=
w_3\circnabla w_8=
w_4\circnabla w_4=
0,
\\
&&
w_4\circnabla w_6=
w_4\circnabla w_7=
w_5\circnabla w_5=
w_5\circnabla w_6=
w_6\circnabla w_6=
w_6\circnabla w_7
=0.
\end{eqnarray*}
We can check that $L = \Span_{\R}\{w_1, w_2, w_3,w_4\}$ is a Dirac structure, and thus  $(E, L)$ constitutes a Manin pair. 
Furthermore, based on Equation \eqref{Dorfman connection from manin pair}, the  only non-trivial terms in the $L$-Dorfman connection on $E/L$ are \begin{eqnarray*} \nabla_{w_4}\overline{w_5}=\frac{C_4}{C-t}(\overline{w_8}-\overline{w_7}),\qquad \nabla_{w_4}\overline{w_8}=\frac{C_4}{C-t}(\overline{w_8}-\overline{w_7}). \end{eqnarray*}
\end{Ex}

\end{document}